\documentclass[a4paper,11pt]{amsart}

\usepackage[english]{babel}
\usepackage{amsmath,amssymb,amsthm, amscd}
\usepackage[alphabetic]{amsrefs}
\usepackage{enumitem}
\usepackage[unicode,pdfencoding=auto,psdextra]{hyperref}
\usepackage{xcolor}  %%  to highlight comments during the writing-up phase
\usepackage{mathrsfs} % for fancy curly letter

\calclayout

\usepackage{kantlipsum} % for text filler

\calclayout

\usepackage{etoolbox}
\apptocmd{\sloppy}{\hbadness 10000\relax}{}{}


\newtheorem{theorem}{Theorem}[section]

\newtheorem{lemma}[theorem]{Lemma}
\newtheorem{proposition}[theorem]{Proposition}
\theoremstyle{definition}
\newtheorem{definition}[theorem]{Definition}

\theoremstyle{remark}
\newtheorem{remark}[theorem]{Remark}

\newcommand{\CC}{\mathbb{C}} % the complex numbers
\newcommand{\RR}{\mathbb{R}} % the real numbers

\newcommand{\NN}{\mathbb{N}}
\newcommand{\camo}{\mathcal{C}_n}

\newcommand{\C}{\mathbb{C}}
\newcommand{\N}{\mathbb{N}}

\newcommand{\R}{\mathbb{R}}

\newcommand{\holo}{\mathcal{O}}

\newcommand{\ssubset}{\subset\joinrel\subset}

\DeclareMathOperator{\id}{id}

\DeclareMathOperator{\aut}{Aut}
\DeclareMathOperator{\vf}{VF}
\DeclareMathOperator{\ham}{Ham}
\DeclareMathOperator{\lie}{Lie}

\newcommand{\saut}{\mathrm{Aut}^\sigma}
\newcommand{\taut}{\mathrm{Aut}^\tau}

\DeclareMathOperator{\diff}{Diff}

\newcommand\restr[2]{{% we make the whole thing an ordinary symbol
  \left.\kern-\nulldelimiterspace % automatically resize the bar with \right
  #1 % the function
  \littletaller % pretend it's a little taller at normal size
  \right|_{#2} % this is the delimiter
  }}

\newcommand{\littletaller}{\mathchoice{\vphantom{\big|}}{}{}{}}

\title[Holomorphic Approximation for Diffeomorphisms]{Holomorphic Approximation for Real Diffeomorphism Groups}

\author[F. Deng]{Fusheng Deng}
\address{School of Mathematical Science \\ University of Chinese Academy of Sciences \\ Beijing, China}
\email{fshdeng@ucas.ac.cn}

\author[G. Huang]{Gaofeng Huang}
\address{Department of Mathematics \\ Harvard University \\  
Cambridge, USA}
\email{ghuang@math.harvard.edu}

\author[F. Kutzschebauch]{Frank Kutzschebauch}
\address{Mathematical Institute \\ University of Bern \\  
Bern, Switzerland}
\email{frank.kutzschebauch@unibe.ch}

\author[E. F. Wold]{Erlend Forn{\ae}ss Wold}
\address{Department of Mathematics \\  University of Oslo \\  
Oslo, Norway}
\email{erlendfw@math.uio.no}

\begin{document}

\subjclass{MSC2020: 32M17 (Primary),  32E30, 32M05, 32M25 (Secondary).}
\keywords{Anders\'en-Lempert theory, automorphism group, Carleman approximation, complexification, density property, diffeomorphism group.}

\begin{abstract}
    We show that every diffeomorphism of $\mathbb{R}^n$ for $n \ge 2$ can be approximated by an automorphism of $\mathbb{C}^n$ in the Whitney $C^k$-topology for any positive integer $k$ using the notion of density property. More precisely we find sufficient conditions for this holomorphic approximation of diffeomorphisms to hold and prove that the split real forms of most linear algebraic groups satisfy these conditions. In the same manner we also show holomorphic approximations of volume-preserving diffeomorphisms for the split real forms of linear algebraic groups equipped with the left-invariant volume form. 
\end{abstract}

\maketitle

\tableofcontents

\nocite{*}

\section{Introduction}
The complexification of real objects and the approximation of real functions by holomorphic functions constitute an important tool in mathematics. A celebrated theorem of the Swedish mathematician Torsten Carleman states that every continuous function on the real line $\R $ can be approximated in the Whitney topology by holomorphic functions on the complex line $\C $. For an overview of such results, particularly in the context of Several Complex Variables, we refer to the survey article by John Erik Fornæss, Franc Forstnerič, and the fourth author \cite{MR4264040}. In the community of Several Complex Variables, this type of approximation is commonly referred to as Carleman approximation.

The step from approximating functions to approximating diffeomorphims is delicate.  For example, approximation of  diffeomorphisms of the real line $\R$ by holomorphic automorphisms of $\C$ is clearly impossible. Diffeomorphism groups are always large whereas the holomorphic automorphism group of $\C$ consists of affine maps only. So a necessary input for holomorphic approximation of diffeomorphism groups are complex manifolds with large holomorphic automorphism groups. A theory about such complex manifolds, the so-called Anders\'en-Lempert theory, has been developed for about 35 years, see the overview article of Forstneri\v c and the third author \cite{MR4440754}.  

In the present paper we  systematically develop the application of the density property, volume density property, and symplectic density property to the approximation of diffeomorphisms, volume-preserving diffeomorphisms and symplectic diffeomorphisms.  The very first result in this direction was proved by Kutzschebauch and Wold in \cite{MR3794890}, where however the approximating holomorphic automorphism is not required to preserve the totally real submanifold $\R^s \subset \C^n$, $s<n$.  Some isolated results in this direction are:  approximation of Hamiltonian diffeomorphisms of $(S^1)^{2n}$, $(S^1)^n \times \R^n$ and $\R^{2n}$ (however in the much coarser compact-open topology and in the case $(S^1)^n \times \R^n \cong T^*(S^1)^n$ for another symplectic form than the canonical form on $T^*(S^1)^n$) by Pierre Berger and Dmitry Turaev \cite{MR4930686}. Remarkably this result was used by Berger \cite{MR5094439} to disprove a conjecture by Birkhoff \cite{MR0006260} about dynamical properties of real-analytic diffeomorphisms of a cylinder $S^1 \times \R$. Birkhoff predicted the absence of   a certain dynamical behavior  for real-analytic diffeomorphisms, however Berger was able to construct it using holomorphic approximation.  It is likely that the methods and results of the present paper can be used to construct real-analytic diffeomorphisms with a certain behavior on more complicated manifolds.

\bigskip 

In the simplest setting, our main results can be stated as follows. 

\begin{theorem} \label{thm: diffeo-Rn}
    Let $\phi$ be a diffeomorphism of $\R^n, n \ge 2$. Then for any positive continuous function $\epsilon$ on $\R^n$ and any natural number $k$, there exists a holomorphic automorphism $\Phi$ of $\C^n$ such that $\Phi(\R^n) = \R^n$ and 
    \[
        \| \Phi(x) - \phi(x) \|_{C^k} < \epsilon(x)  \quad \text{ for every } x \in \R^n. 
    \]
\end{theorem}

\begin{remark}
    Every diffeomorphism $\phi$ of $\R^n$ is either smoothly isotopic to the identity or to $(x_1, \dots, x_n) \mapsto (-x_1, x_2, \dots, x_n)$. Thus the preceding theorem is a direct consequence of Theorem \ref{Carleman} below. 
    Indeed, consider the following isotopy of diffeomorphisms
    \begin{align} \label{Isotopy-Id}
        \phi_t(x) = \begin{cases}
            t \phi(0) + \frac{\phi(tx)-\phi(0)}{t}, & t \in (0,1] \\
            d_0 \phi(x), & t = 0. 
        \end{cases} 
    \end{align}
    By Hadamard's lemma we have $\phi(tx) - \phi(0) = t \sum_i x_i g_i(tx)$ for smooth maps $g_i$. Thus $\phi_1=\phi$ is smoothly joint to $\phi_0 = d_0 \phi \in \mathrm{GL}_n(\R)$. Since $\mathrm{GL}_n(\R)$ has two connected components, we can  connect further either to the identity or $(x_1, \dots, x_n) \mapsto (-x_1, x_2, \dots, x_n)$.   We concatenate these isotopies smoothly by a standard endpoint-stationary reparametrization (i.e.\@ varying the time dependent vector fields to zero (in space) from both sides of the transition time). 
\end{remark}

Let $\omega = d x_1 \wedge \cdots \wedge d x_n$ be the standard volume form on $\R^n$ and $\Omega = d z_1 \wedge  \cdots \wedge d z_n$ the standard holomorphic volume form on $\C^n$. By the same isotopy \eqref{Isotopy-Id} every volume-preserving diffeomorphism of $\R^n$ can be smoothly joint to an element of $\mathrm{SL}_n(\R)$ which is connected, and thus to the identity.  Our second main result in this simple setting is the following. 
\begin{theorem} \label{thm: Vdiffeo-Rn}
    Let $\phi$ be a diffeomorphism of $\R^n, n \ge 2$, and $\phi^* \omega = \omega$. Then for any positive continuous function $\epsilon$ on $\R^n$ and any natural number $k$, there exists a holomorphic automorphism $\Phi$ of $\C^n$ such that $\Phi^* \Omega = \Omega$, $\Phi(\R^n) = \R^n$  and 
    \[
        \| \Phi(x) - \phi(x) \|_{C^k} < \epsilon(x)  \quad \text{ for every } x \in \R^n. 
    \]
\end{theorem}

Our third main result for the symplectic case in the simple setting of $\R^{2n}$ was proved by the first and the fourth named authors in \cite{MR4423269}. 
\medskip 

Before giving our more general results we introduce the relevant definitions. The notion of real holomorphic objects is natural since we aim to build holomorphic automorphisms preserving a real submanifold.

\begin{definition} \label{def: realHolo}
    Let $X$ be a complex manifold with an antiholomorphic involution $\sigma \colon X \to X$.  
    A holomorphic vector field $V$ on $X$ is called \textbf{real holomorphic with respect to}  $\sigma$, 
    if $\sigma_* (V) = \overline{V}$. 
    A holomorphic function $f \in \holo(X)$ is called real holomorphic with respect to $\sigma$,  if $\sigma^\ast f= \overline{f}$ and likewise, a holomorphic differential form $\alpha$ is real holomorphic with respect to $\sigma$, if $\sigma^\ast \alpha = \overline{\alpha}$. 
\end{definition}

\begin{definition}\label{definition: realform}
(1) Let $X_\R$ be a smooth manifold of dimension $n \ge 1$. A \textbf{complexification} $(X, \sigma, \jmath)$ of $X_\R$ is a complex manifold $X$ of complex dimension $n$, together with an antiholomorphic involution $\sigma \colon X \to X$ and a smooth map $\jmath \colon X_\R \to X$, such that $\jmath$ is a proper embedding and the image of $X_\R$ is a union of connected components of the fixed-point set of $\sigma$. 

(2)
Let $X_\R$ be a smooth manifold of dimension $n \ge 1$ with a real volume form $\Omega_\R$. A \textbf{volume complexification} $(X, \Omega, \sigma, \jmath)$ of $(X_\R, \Omega_\R)$ is a complexification $(X, \sigma, \jmath)$ of $X_\R$, together with a holomorphic volume form $\Omega$ satisfying
\begin{enumerate}[label=(\roman*)]
	%\item		$(X, \sigma)$ is a complexification of $X_\R$.
	\item		The pullback $\jmath^\ast \Omega$ coincides with $\Omega_\R$. 	
	\item		The volume form $\Omega$ is real holomorphic.
\end{enumerate}

(3)
Let $X_\R$ be a smooth manifold of dimension $2n, \, n \ge 1$, with a real symplectic form $\omega_\R$. A \textbf{symplectic complexification} $(X, \omega, \sigma)$ of $(X_\R, \omega_\R)$ is a complexification $(X, \sigma, \jmath)$ of $X_\R$, together with a holomorphic symplectic form $\omega$ satisfying
\begin{enumerate}[label=(\roman*)]
	\item		The pullback $\jmath^\ast \omega$ coincides with $\omega_\R$. 	
	\item		The symplectic form $\omega$ is real holomorphic.
\end{enumerate}
\end{definition}

We use the following form of density property introduced by Dror Varolin \cite{MR1829353}. 
\begin{definition} \label{def: sigmaDPs}
    Let $X$ be a complex manifold with an antiholomorphic involution $\sigma \colon X \to X$. Let $\mathfrak{g}$ be a Lie subalgebra of all holomorphic vector fields on $X$. We say that $\mathfrak{g}$ has the $\sigma$-\textbf{density property} if the complex Lie algebra generated by $\C$-complete, real holomorphic vector fields in $\mathfrak{g}$ is dense in $\mathfrak{g}$ in the compact-open topology. 
\end{definition}

    We consider the following Lie algebras of holomorphic vector fields in the three settings (1), (2) and (3) of Definition \ref{definition: realform}
    \begin{align*}
        \vf (X) &= \{ \text{holomorphic vector fields on } X \}, \\
        \vf_{\Omega}(X) &= \{ \text{holomorphic volume-preserving vector fields on } X \}, \\
        \vf_{\omega}(X) &= \{ \text{holomorphic symplectic vector fields on } X \}, \\
        \vf_{\rm ham}(X) &= \{ \text{holomorphic Hamiltonian vector fields on } X \}.
    \end{align*}
    We will use the notions $\sigma$-density property, $\sigma$-volume density property, $\sigma$-symplectic density property and $\sigma$-Hamiltonian density property, if the $\sigma$-density property holds for the Lie algebras $\vf(X), \vf_{\Omega}(X), \vf_{\omega}(X)$ and $\vf_{\rm ham}(X)$. 

\begin{remark} \label{replace-Lie-by-Sum}
    (1) If the Lie subalgebra $\mathfrak{g}$ is closed in the compact-open topology, then the $\sigma$-density property of $\mathfrak{g}$ is equivalent to the fact that the complex vector space generated by complete, real holomorphic vector fields in $\mathfrak{g}$ is dense in $\mathfrak{g}$ in the compact-open topology. 
    Indeed $[V, W] = \lim_{t \to 0} (\varphi_t^* W - W)/t$ uniformly on compact subsets of $X$, where $\varphi_t$ denotes the time-$t$ flow of $V$. Note that the pullback $\varphi_t^* W$ is in $\mathfrak{g}$ by closedness. 

    (2) When $\mathfrak{g}=\vf_{\rm ham}(X)$, the holomorphic functions modulo locally constant functions $\holo(X)/\CC$ with the Poisson bracket is (anti-)isomorphic to the Lie algebra $\vf_{\rm ham}(X)$ of Hamiltonian vector fields on $X$. Thus the $\sigma$-density property of $\vf_{\rm ham}(X)$ is equivalent to the fact that the complex vector space generated by real holomorphic Hamiltonian functions inducing complete Hamiltonian vector fields is dense in $\holo(X)/\CC$ in the compact-open topology.
\end{remark}

    Let $\alpha$ be a real holomorphic differential form on $(X, \sigma)$. 
    We denote by $\saut_{\alpha} (X)$ the group of $\sigma$-equivariant, $\alpha$-preserving holomorphic automorphisms of $X$:
    \begin{align*}
        \saut_{\alpha} (X) = \{\, \Phi \in \aut(X) : \Phi \circ \sigma = \sigma \circ \Phi, \, \Phi^* \alpha = \alpha \,\}.
    \end{align*}
    %By $\sigma$-equivariance, these automorphisms preserve the fixed-point set of $\sigma$. 
    Since $\alpha$ is real holomorphic with respect to $\sigma$, 
    \[
        \saut_\alpha (X) = \{\, \Phi \in \aut(X) : \Phi (\mathrm{Fix}(\sigma)) = \mathrm{Fix}(\sigma), \, \Phi^* \alpha = \alpha \,\}.
    \]

\begin{definition}
    Let $X_\R$ be a smooth manifold and $G_\R$  a subgroup of $\diff(X_\R)$. We say that $G_\R$ is \textbf{holomorphically approximable}, if there exist a complexification $(X, \sigma)$ of $X_\R$ and a subgroup $G$ of $\saut (X)$, such that $G \rvert_{X_\R} \subset G_\R$ and $G \rvert_{X_\R}$ is dense in $G_\R$ in the Whitney topology. 
\end{definition}

    We consider the following diffeomorphism groups in the three settings (1), (2) and (3) of Definition \ref{definition: realform}
    \begin{align*}
        \diff (X_\R) &= \{ \text{diffeomorphisms of } X_\R \}, \\
        \diff_{\Omega_\R}(X_\R) &= \{ \text{volume-preserving diffeomorphisms of } X_\R \}, \\
        \diff_{\omega_\R}(X_\R) &= \{ \text{symplectic diffeomorphisms of } X_\R \}, \\
        \ham_{\omega_\R}(X_\R) &= \{ \text{Hamiltonian diffeomorphisms of } X_\R \}.
    \end{align*}
    A Hamiltonian diffeomorphism is a symplectic diffeomorphism as the time-1 flow of a time-dependent Hamiltonian vector field, see Section \ref{sec: Ham-approximation}.
    \medskip

Here come our general results. They are proved in Section \ref{general-results} as Theorem \ref{Global-Carleman}, Theorem \ref{Global-Carleman-vol} and Theorem \ref{Global-Carleman-ham}.
\begin{theorem} \label{Carleman}
    If a smooth manifold $X_\R$ has a Stein complexification $(X, \sigma)$ with the $\sigma$-density property, then the subgroup of the diffeomorphism group $\diff(X_\R)$ consisting of elements smoothly isotopic to the identity is holomorphically approximable. 
\end{theorem}

\begin{theorem} \label{Carleman-vol}
    If a smooth manifold $X_\R$ with a real volume form $\Omega_\R$ has a Stein volume complexification $(X, \Omega, \sigma)$ with the $\sigma$-volume density property, then the subgroup of the volume-preserving diffeomorphism group $\diff_{\Omega_\R}(X_\R)$ consisting of elements smoothly isotopic to the identity is holomorphically approximable.
\end{theorem}

\begin{theorem} \label{Carleman-ham}
    If a smooth manifold $X_\R$ with a real symplectic form $\omega_\R$ has a Stein Hamiltonian complexification $(X, \omega, \sigma)$ with the $\sigma$-Hamiltonian density property, then the Hamiltonian diffeomorphism group $\ham_{\omega_\R}(X_\R)$ of $X_\R$ is holomorphically approximable.
\end{theorem}

In the last section we give  concrete examples where the assumptions of the approximation theorems \ref{Global-Carleman}, \ref{Global-Carleman-vol} and \ref{Global-Carleman-ham} are fulfilled. The new results are Theorem \ref{G-Diffeo} and Theorem \ref{G-vol-Diffeo} on linear algebraic groups.

Furthermore we give criteria for the $\sigma$-density property and the $\sigma$-volume density property, Theorem \ref{sigma-DP} and Theorem \ref{sig-VDP}, based on similar known criteria for density properties developed by Kaliman and Kutzschebauch in \cites{MR2385667, MR3492044}. For the Hamiltonian case we give the list of up-to-date known examples, whereas criteria as in the density and volume density cases are still unavailable.

\section{Preliminaries}

Let $X$ and $Y$ be smooth manifolds and $p \in X$. Consider smooth maps $f, g \colon X \to Y$ with $f(p)=g(p)=q$. We say that $f$ has first order contact with $g$ at $p$ if $df_p = dg_p$ as maps of $T_pX \to T_q Y$; and $f$ has $k$th order contact with $g$ at $p$ if $df \colon TX \to TY$ has $(k-1)$th order contact with $dg$ at every point in $T_p X$. This gives an equivalence relation for smooth maps in $C^\infty(X,Y)$ for each positive integer $k$. Denote by $J^k(X,Y)_{p,q}$ the set of equivalence classes of maps $f \colon X \to Y$ with $f(p)=q$. The disjoint union $J^k(X,Y) := \cup_{(p,q) \in X \times Y} J^k(X,Y)_{p,q}$ may be given the structure of a smooth manifold. 

Every smooth map $f \colon X \to Y$ induces a smooth map $j^kf \colon X \to J^k(X,Y)$ defined by 
\[
    j^kf(p)= [f] \in J^k(X,Y)_{p,f(p)} \text{ for every } p \in X. 
\]
The family of sets $\{ f \in C^\infty(X,Y): j^kf(X) \subset U \}$ where $U$ is an open subset of $J^k(X,Y)$ form a basis for a topology on $C^\infty(X,Y)$, which is called the Whitney $C^k$-topology, see Golubitsky--Gullemin \cite{MR0341518}. By choosing Riemannian metrics on the manifolds $X$ and $Y$ one can equip the fibers of the jet space with norms, i.e. the jet space $J^k(X,Y)$ with a Finsler metric. Over a compact $X$ the jet space $J^k(X,Y)$  is then equipped with
a norm (the $C^k$-norm) compatible with the Whitney $C^k$-topology. By using a sequence of exhausting compacts of $X$  one gets a metric $d_k$ on the manifold $J^k(X,Y)$ compatible with this topology. 
\bigskip 

The proof of the approximation theorems \ref{Carleman}, \ref{Carleman-vol}, \ref{Carleman-ham}, goes by induction, where at each induction step we use the corresponding  $\sigma$-density property. We need an appropriate exhaustion of $X$ by $\sigma$-invariant, holomorphically convex compact subsets $S_j, j=1, 2,\dots$. The forthcoming considerations will be used to construct these subsets. 

\begin{lemma}
    Every Stein manifold $X$, endowed with an antiholomorphic involution $\sigma \colon X \to X$, admits a holomorphic embedding $G \colon X \to \C^q$ for some $q \in \N$ such that $G(X_\R) = G(X) \cap \R^{q}$. 
\end{lemma}
\begin{proof}
    Since $X$ is Stein, there are holomorphic functions $f_i \colon X \to \C$ such that
\[
    F \colon X \hookrightarrow \C^N, x \mapsto (f_1(x), f_2(x), \dots, f_N(x))
\]
is a proper holomorphic embedding. 
Then, we can choose $q = 2N$ and embed $X$ into $\C^{2N}$ as
\begin{align*}
    G \colon X \hookrightarrow \C^{2N}, x \mapsto (g_1(x), g_2(x), \dots, g_{2N}(x)),
\end{align*}
where
\begin{align*}
    g_{2k-1} = f_k + \overline{f_k \circ \sigma} \, \text{ and }  \, g_{2k} = i(f_k - \overline{f_k \circ \sigma} ), \, k = 1, 2, \dots, N, \, \, \text{ are real holomorphic.} 
\end{align*}
Since $G (\sigma(p)) = \overline{G(p)}$ for $p \in X$, the fixed-point set $X_\R$ is being embedded into $\R^{2N} \subseteq \C^{2N}$.  Moreover, a simple calculation, using that $F$ is injective, shows that $G(X_\R) = G(X) \cap \R^{2N}$.
\end{proof}

From now on, we view $X$ as a subvariety of $\C^{q}$ for some $q \in \N$ and $\sigma$ as the restriction of complex conjugation of $\C^{q}$ to $X$.  

\begin{definition} \label{sublevelset}
For $p_0$ in $\R^q$, the function
$\C^{q} \to \R_{\ge 0}, \, p \mapsto \| p - p_0 \|^2$ is $\sigma$-invariant and thus its restriction $\rho \colon X \to \R_{\ge 0}$ to $X$  is a $\sigma$-invariant strictly plurisubharmonic exhaustion function on $X$. 
We consider the closed sublevel set 
	\begin{align*} 
		Z_r = \{ p \in X: \rho(p) \le r \} \quad \text{ and } \quad	Z_r^\R = Z_r \cap X_\R	
	\end{align*}
the intersection of $Z_r$ with the fixed-point set. 
\end{definition}

\begin{lemma} \label{Saturn-holoconv}
    The union $S:= Z_{r_1}^\R \cup Z_{r_2}$ is holomorphically convex in $X$ for $r_1, r_2 \in  \R_{\ge 0}$. 
\end{lemma}
\begin{proof}
   % The only nontrivial case is $r_1 > r_2 > 0$. 
    By a result of Evgenii Chirka and Maxim Smirnov \cite{MR1155560}*{Theorem 2}, the union 
    \[
        \{ p \in \R^{q}: \| p - p_0 \|^2 \le r_1 \} \cup \{ p \in \C^{q}: \| p - p_0 \|^2 \le r_2 \}
    \]
    is polynomially convex in $\C^{q}$. Since $X$ is a complex subvariety of $\C^{q}$, its intersection $ Z_{r_1}^\R \cup Z_{r_2}$ with $X$ is holomorphically convex in $X$ by Cartan's Theorem B. 
\end{proof}

\section{Holomorphic Approximations via Density Properties}
\label{general-results}

\subsection{The Diffeomorphism Group}

Here comes the local version of the complexification result Theorem \ref{Carleman},
an approximation of Anders{\'e}n-Lempert type for diffeomorphisms of $X_\R$ uniformly on the compact subsets $S$. 
\begin{theorem} \label{theorem: LocalCarleman}
	Let $X_\R$ be a smooth manifold which admits a Stein complexification $(X, \sigma)$. Let $\varphi \colon [0,1] \times X_\RR \to X_\RR$ be a smooth isotopy of diffeomorphisms and $r \ge 0$ such that $\varphi_0 = \id$ and $\restr{\varphi_t}{Z_r^\RR} = \id_{Z_r^\RR}$ for all $t$ in $[0,1]$.
    
    If $X$ has the $\sigma$-density property, then for any $k \in \N$, any $\epsilon>0$, any $b \in (0,r)$ and any $a > b$, there exists a smooth isotopy of automorphisms $\Phi \colon [0,1] \times X \to X$ such that $\Phi_{t} \in \saut(X)$ and 
	\begin{align*}
		  d_k( j^k \Phi_{t}(p) - j^k \varphi_t(p)) &< \epsilon \text{ for all } p \in Z_{a}^\R, 	\\
		d_k( j^k \Phi_{t} (p) - j^k \mathrm{id} (p) )		&< \epsilon  \text{ for all } p \in Z_{b}, \quad \text{ for all } t \in [0,1].	 %\label{ApproxOnZb}
	\end{align*}
\end{theorem}
\begin{proof}
    We introduce an auxiliary time variable $s$
    \[
        \psi \colon [0,1] \times [0,1] \times X_\R \to X_\R, (s, t, p) \mapsto \varphi(st, p)
    \]
    and consider $\varphi_t$ as the time-1 map of the time-$s$ dependent vector field 
    \[
        V \colon [0,1] \times [0,1] \times X_\R \to T X_\R 
    \]
    defined by 
	\begin{align*} %\label{equation: V_t}
		\frac{d } {d s} \varphi_{st} (p) = V (s,t, \varphi_{st}(p)) = V_{s,t}(\varphi_{st}(p)) , \quad s, t \in [0,1], \, p \in X_\R,		
	\end{align*}
    with the initial condition $\psi(0,t,p) = \varphi(0,p) = p$. 
    Since $\restr{\varphi_t}{Z_r^\RR} = \id_{Z_r^\RR}$ for $t \in [0,1]$, we can extend $V_{s,t}$ to be 0 in an open neighborhood of $Z_{b'}$ for $s, t \in [0,1]$, where $b' \in (b,r)$. 
    Choose $a'>0$ such that 
    \begin{align*}
        \bigcup_{ t \in [0,1] }  \varphi_t ( Z_{a}^\RR ) \subset Z_{a'}^\RR. 
    \end{align*}
    Let $Q$ be a compact Hausdorff space whose points we denote by $q$. Given a continuously $q$-dependent function $h_q$ on $S= Z^\R_{a'} \cup Z_{b'}$ which is in $C^{k}(S) \cap \holo(Z_{b'})$, by the holomorphic convexity of $S$ in $X$ (Lemma \ref{Saturn-holoconv}) and the parametric version of the Mergelyan theorem by Forn{\ae}ss--Forstneri{\v c}--Wold \cite{MR4264040}*{Theorem 20}, $h_q$ is approximable in the $C^{k}$-norm by a continuously $q$-dependent function $P_q$ which is holomorphic in a neighborhood $\Omega$ of $S$. 
    The same is true for sections in the holomorphic tangent bundle since any holomorphic vector bundle on a Stein space is stably trivial, for details see e.g.\@ Forstneri{\v c} \cite{MR2975791}*{p.\@ 64} and Giraldo--S{\'a}nchez-Arellano \cite{MR4904178}*{Theorem 4.9}. Hence $V_{s,t}$ can be approximated in the $C^k$-norm on $S$ by a continuously $(s,t)$-dependent holomorphic vector field $W_{s,t}$ for $s,t \in [0,1]$.
    Moreover, we can assume that $W(s,t,p)$ is a finite sum of the form $\sum_{n,m} s^n t^m W_{nm}(p)$ by a patching argument as in the proof of \cite{MR1314745}*{Lemma 1.2}. 

    Then by 
    the Cartan--Oka--Weil theorem, we can approximate $W_{nm}$ on a relative compact Runge open neighborhood of $S$ by a holomorphic vector field $\widetilde{W}_{nm}$ on $X$. Since both $W_{nm}$ and $\widetilde{W}_{nm}$ are  holomorphic in a neighborhood of $S$, the approximation holds in the $C^{k}$-norm.

    Since $V_{s,t}$ is a real vector field on $X_\R$, we can replace $\widetilde{W}_{nm}$ by $(\widetilde{W}_{nm} + \overline{\sigma_* \widetilde{W}_{nm}})/2$ such that 
    \[
        \widetilde{W}_{s,t}(p) = \widetilde{W}(s,t,p) = \sum_{n,m} s^n t^m \widetilde{W}_{nm}(p)
    \]
    is real holomorphic with respect to $\sigma$ and approximates $V_{s,t}$. 

    By the $\sigma$-density property and Remark \ref{replace-Lie-by-Sum}, $\widetilde{W}_{nm}$ can be approximated on $\Omega$ by a complex linear combination $L$ of complete real holomorphic vector fields. 
    Since the real flows of real holomorphic vector fields preserve $X_\R$, we may assume that the coefficients in this linear combination are real. Indeed we can separate the coefficients of $L$ and thus also $L$ into real and imaginary parts $\widetilde{W}_{nm} \approx L_1 + i L_2$. Acting on both sides with $\overline{\sigma_*(\cdot)}$, we get $\widetilde{W}_{nm} \approx L_1 - i L_2$ since $\widetilde{W}_{nm}$ and $L_1, L_2$ are real holomorphic. Thus $L_2 \rvert_{\Omega} \approx 0$ and $L_1$ approximates $\widetilde{W}_{nm}$ on $\Omega$. Hence $\widetilde{W}_{s,t}$ can be approximated on $\Omega$ (thus $V_{s,t}$ on $S$) by a finite sum $\sum_{l=1}^N p_l(s,t) \widetilde{W}_l$ where $p_l(s,t)$ is a polynomial and $\widetilde{W}_l$ is a complete real holomorphic vector field on $X$.

    By definition for all $t \in [0,1]$ the vector field  $V_{s,t}$ is integrable for time $s \in [0,1]$ for all starting points $x \in Z_a^\R \cup Z_b$ and its flow is contained in $S = Z^\R_{a'} \cup Z_{b'} \subset \Omega$, thus there is an open neighborhood $U$ containing $Z_a^\R \cup Z_b$ such that $\sum_{l=1}^N p_l(s,t) \widetilde{W}_l$ is integrable for time $s \in [0,1]$ for all initial points $x$ in $U$. 

    For a fixed $n \in \NN$, $j \in \{0, 1, \dots, n-1\}$, we denote by $\Phi^{(s, l)}_\theta$ the time-$\theta$ flow map of the time-$s$ independent vector field $ p_l(s,t) \widetilde{W}_l$, $l \in \{ 1, \dots, N\}$ and consider the composition
    \begin{align*} %\label{algorithm}
        \Phi^{(s)}_{\theta} =  \Phi^{(s, 1)}_{\theta} \circ \Phi^{(s, 2)}_{\theta} \circ \cdots \circ \Phi^{(s,N)}_{\theta} .
    \end{align*}
    Now we take 
    \begin{align} \label{composition}
		 \Phi^{(\frac{n-1}{n})}_{1/n} \circ  \Phi^{(\frac{n-2}{n})}_{1/n}  \circ \cdots \circ  \Phi^{(\frac{1}{n})}_{1/n} \circ  \Phi^{(0)}_{1/n} 
	\end{align}
    which approximates the time-$1$ flow of $\sum_{l=1}^N p_l(s,t) \widetilde{W}_l$ uniformly on a relatively compact open neighborhood of any compact subset in $U$ for large enough $n$ by Abraham--Marsden \cite{MR0515141}*{Theorem 2.1.26}. Indeed an algorithm for the vector field $\left( 1, \sum_{l=1}^N p_l(s,t) \widetilde{W}_l \right) $ on $[0,1] \times ([0,1]\times X)$ is given by 
    \begin{align*}
        [0, \epsilon) \times [0,1]^2 \times X \to  [0,1]^2 \times X, \quad (\theta, s, t, p) \mapsto (s + \theta, t, \Phi^{(s)}_\theta(p))  
    \end{align*}
    whose $n$th iterate after projecting to $X$ is \eqref{composition}. 
    By the Cauchy estimate, the approximation is in the $C^k$-norm and \eqref{composition} depends smoothly on $t$. 
    
    Finally, since approximation of $V_{s,t}$ by the vector field $\sum_{l=1}^N p_l(s,t) \widetilde{W}_l$ implies the approximation of $\varphi_t$ by the time-1 flow of $\sum_{l=1}^N p_l(s,t) \widetilde{W}_l$, it follows that \eqref{composition} approximates $\varphi_t$ on a relatively compact open neighborhood of $Z^\R_a \cap Z_b$. 
\end{proof}

\begin{remark} \label{Local-Carleman-w/o-interpolation}
    The same proof also shows that for any $a \ge 0$ and any smooth isotopy of diffeomorphisms $\varphi \colon [0,1] \times X_\RR \to X_\RR$, the $\sigma$-density property of $X$ implies the existence of a smooth isotopy of automorphisms $\Phi \colon [0,1] \times X \to X$ such that $\Phi_{t} \in \saut(X)$ and $d_k( j^k \Phi_{t}(p) - j^k \varphi_t(p)) < \epsilon$ for all $p \in Z_{a}^\R$ and $t \in [0,1]$.
\end{remark}

\begin{theorem} \label{Global-Carleman}
    Let $X_\R$ be a smooth manifold which admits a Stein complexification $(X, \sigma)$. If $X$ has the $\sigma$-density property, then for every diffeomorphism $\varphi \in \diff(X_\R)$ smoothly isotopic to $\id$, for any $k \in \N$ and any positive continuous function $\epsilon$ on $X_\R$, there exists a holomorphic automorphism $\Phi \in \saut(X)$ such that
	\[	
        d_k( j^k \Phi (p) - j^k \varphi(p) ) < \epsilon (p) \quad \text{ for all }\, p \in X_\R.	
    \]
\end{theorem}

\begin{proof}
    By assumption there exists a smooth isotopy $\varphi_t$ in $\diff(X_\R), t \in [0,1]$ with $\varphi_0 = \mathrm{id}$ and $\varphi_1= \varphi$. 
	%\smallskip
	Given any natural number $j$ and any positive $\varepsilon_j$, assume that we have a holomorphic automorphism $\Phi_j \in \saut(X)$, a smooth isotopy $\psi_{j,t}$ of diffeomorphisms of $X_\R$ with $\psi_{j,0}=\id$, real numbers $r_j, s_j$ with $r_j \ge j-1, s_j \ge r_j + 1$ such that
	\smallskip
	\begin{enumerate}
		\item[($1_j$)] 	The image of $ Z_{r_j} $ under $\Phi_j$ is contained in $Z_{s_j}$. 
		\item[($2_j$)]	For $j \ge 2$ 
					\[	 d_k (j^k \Phi_j (p) - j^k \Phi_{j-1}(p) ) < \varepsilon_j \quad  \forall \, p \in Z_{r_{j-1}}.
                    \] %, \quad  \| \Phi^{-1}_j - \Phi^{-1}_{j-1} \|_{C^k(Z_{s_{j-1}})} < \varepsilon_j		\]
		\item[($3_j$)]	On $Z^\R_{s_j }$ we have $\psi_{j,t}$ is the identity for $t$ in $[0,1]$. %\textcolor{red}{why $\varepsilon_j$?}
		\item[($4_j$)]	For all $p$ in $X_\R$
					\[	d_k (j^k (\psi_{j, 1} \circ \Phi_j) (p) - j^k \varphi (p) ) < \epsilon (p).
                    \]
	\end{enumerate}

    \textbf{Induction base}: For $j = 1$ take $p_0 \in X_\R$ and \textcolor{black}{$r_1 = 0$}. In this case we have 
	\[
		Z_0 = Z_0^\RR =  \{ p_0 \} \subset X_\RR
	\]
	from Definition \ref{sublevelset} and by the choice of the strictly plurisubharmonic exhaustion function. 
	Choose \textcolor{black}{$s_1 \ge 1$} so that 
	\begin{align} \label{equation: S_1}
		\varphi_1 \left( Z_{ 1}^\RR \right) \ssubset Z_{s_1}^\RR.
	\end{align}
	Choose $T > s_1 $. Then by \eqref{equation: S_1}
	\begin{align}
		Z_{0}^\RR \subset Z_1^\RR  \ssubset (\varphi_1)^{-1} \left( Z_{s_1}^\RR \right) \subset (\varphi_1)^{-1} \left( Z_{T}^\RR \right).	\label{equation: R_1}	
	\end{align}
	Using Remark \ref{Local-Carleman-w/o-interpolation} we get a smooth isotopy $A_t \in \saut(X)$ approximating $\varphi_t$ on the compact subset 
	\[
		\bigcup_{ t \in [0,1]} (\varphi_t)^{-1} \left( Z_{T + 3 }^\RR \right)
	\]
	of $X_\RR$. Since $A_t$ approximates $\varphi_t$ on this compact, we have 
	\begin{align}
		(A_t)^{-1} \left( Z_{T + 2}^\RR \right) 	\ssubset  	(\varphi_t)^{-1} \left( Z_{T + 3}^\RR \right).  \label{equation: r_3}
	\end{align}
	In particular, $A_1$ approximates $\varphi_1$ on $A_1^{-1} \left( Z_{T + 2}^\RR \right) $, hence $\varphi_t \circ \restr{A_t^{-1}}{X_\RR}$ is close to the identity on $Z_{T + 2}^\RR$. Choose \textcolor{black}{$\Phi_1 = A_1$}.
	
	%\smallskip
	To construct an isotopy $\psi_{1,t}$ we interpolate between the identity on $Z_{T}^\RR $ and $\varphi_t \circ \restr{A_t^{-1}}{X_\RR}$ outside $Z_{T + 2}^\RR$. More precisely, let $W_t$ be the infinitesimal generator of the isotopy $\varphi_t \circ \restr{A_t^{-1}}{X_\RR}$. Fix $\psi_{1,t}$ to be the identity on $Z_{T}^\RR$ by multiplying $W_t$ with a cutoff function $\gamma$, which is zero on $Z_{T}^\RR$ and one outside $ Z_{T + 1}^\RR $. %Notice that these are still subsets of $\camo^\RR$ because $A_1$ preserves $\camo^\RR$. 
	Namely, we take \textcolor{black}{$\psi_{1,t}$} to be the time-$t$ flow of the vector field $\gamma W_t$. 
	
	By the choice of $T$ we have that $Z_{s_1}^\RR$ is contained in $Z_{T}^\RR$, which implies that $\psi_{1,t}$ is the identity on $Z_{s_1}^\RR$ for all $t$ in $[0,1]$. This shows $(3_1)$. 
	\smallskip 
	
	To see that $(4_1)$ is satisfied, let $p \in X_\RR$ and consider separately
	\begin{enumerate}[label=(\roman*)]
		\item		$p \in A_1^{-1} \left( Z_{T}^\RR \right)$: $\psi_{1,1}$ is the identity at $A_1 (p)$ and $A_1$ approximates $\varphi_1$ by \eqref{equation: r_3}. 
		\smallskip
		\item		$p \notin A_1^{-1} \left( Z_{T + 2}^\RR \right)$: $\psi_{1,1}$ is equal to $\varphi_1 \circ A_1^{-1}$ at $A_1 (p)$ by the choice of $\gamma$.
		\smallskip
		\item		$p \in A_1^{-1} \left( Z_{T + 2}^\RR \setminus Z_{T}^\RR \right)$: $A_1$ approximates $\varphi_1$ by \eqref{equation: r_3} and $\psi_{1,1} \circ \varphi_1$ is the interpolation between $\mathrm{id} \circ \varphi_1$ 				and $\varphi_1 \circ A_1^{-1} \circ \varphi_1$. Here $\varphi_1 \sim A_1$ implies $A_1^{-1} \circ \varphi_1 \sim \mathrm{id}$.  
	\end{enumerate}
	
	Last, let us check that $(1_1)$ holds. By the choice of $A_t$, we may assume that there exists a small positive $\delta$ which is less than one, such that 
	\[
		\Phi_1 \left( Z^\RR_{0} \right) \subset 	\varphi_1 \left( Z_{ \delta}^\RR \right) \subset \varphi_1 \left( Z_{ 1}^\RR \right) \ssubset Z_{s_1}^\RR .
	\]
	The first inclusion follows from the fact that $\Phi_1 = A_1$ approximates $\varphi_1$ by \eqref{equation: R_1} and the last inclusion is due to the choice of $s_1$ in $\eqref{equation: S_1}$. This concludes the induction base. 
    \bigskip

    \textbf{Induction step}: Take \textcolor{black}{$r_{j+1} = s_j + 1$} and choose $a > \max \{ r_j + 1, r_{j+1} \}$ such that
	\begin{align}
		%\bigcup_{ t \in [0,1]} \psi_{j,t} (Z^\RR_{R_{j+1}}) &\ssubset Z^\RR_{a} 	\nonumber	\\
		%\bigcup_{ t \in [0,1]}  (\psi_{j,t})^{-1}( Z^\RR_{R_{j+1}}) &\ssubset  Z^\RR_{a} 	\nonumber \\
		\Phi_j ( Z_{r_{j+1}} ) \ssubset	Z_a. \label{equation: settinga}
	\end{align}	
	By Theorem \ref{theorem: LocalCarleman} there exists a smooth isotopy $\Psi_{j,t}$ in $\saut(X)$, which approximates the identity near $Z_{s_j}$ and $\psi_{j,t}$ near $Z^\RR_{a + 2}$.  % Using $R = S_j + \varepsilon_j$, $a = a +2$, $b=S_j$
	Thus %the symplectic diffeomorphism
	\[	
		\xi_{j,t} = ( \restr{\Psi_{j,t}}{X_\RR} )^{-1} \circ \psi_{j,t}	
	\]
	approximates the identity on $Z_{a+2}^\RR$. %Being the composition of two Hamiltonian isotopies, $\sigma_{j,t}$ is also a Hamiltonian isotopy. 
	
	%\smallskip
	Moreover take a cutoff function $\chi$ on $X_\RR$ such that it is zero on $Z^\RR_a$ and equal to one outside $Z_{a+1}^\RR$. Let $V_t$ be the infinitesimal generator of the smooth isotopy $\xi_{j,t}$ and let $\tilde{\xi}_{j,t}$ be the flow map of the vector field $\chi V_t$. Then $\tilde{\xi}_{j,t}$ is the identity on $Z^\RR_a$, close to the identity on $Z^\RR_{a+2}$, and equal to $\xi_{j,t}$ outside $Z^\RR_{a+2}$.
	
	%\smallskip
	Then we have on $X_\RR$
	\begin{align}
		\Psi_{j,t} \circ \tilde{\xi}_{j,t}  &\sim  \psi_{j,t},  \label{equation: psi_jt} \\ 
		(\Psi_{j,t} \circ \tilde{\xi}_{j,t})^{-1} &\sim (\psi_{j,t})^{-1},  \nonumber
	\end{align}	
	by the choices of $\Psi_{j,t}, \xi_{j,t}, \tilde{\xi}_{j,t}$.
	Next, choose \textcolor{black}{$s_{j+1} > a $} so that 
	\begin{align} \label{equation: Sj+1FromPsi}
		\Psi_{j,1} ( Z_a ) \ssubset Z_{s_{j+1}}. 	
	\end{align}
	Set $b = s_{j+1} + 1$ and pick $c$ and $d$ so that
	\begin{align} \label{equation: settingc}
		Z_{b+2}^\RR &\ssubset \varphi \left( Z_c^\RR \right), \\
		\Phi_j   \left(  Z_c^\RR \right) &\ssubset Z_d^\RR, \label{equation: settingd1} \\
		Z_{c}^\RR &\ssubset \psi_{j,1} \left( Z_d^\RR \right). \label{equation: settingd2}
	\end{align}
	
	Next, apply Theorem \ref{theorem: LocalCarleman} to obtain an isotopy $\Xi_{j,t} \in \saut(X)$, which approximate $\tilde{\xi}_{j,t}$ on $Z^\RR_{d}$ and the identity on $Z_a$. Set \textcolor{black}{$\Phi_{j+1} = \Psi_{j,1} \circ \Xi_{j,1} \circ \Phi_j$}. The above choices of $c$ and $d$ allow us to approximate on $Z_c^\RR$
	\begin{align*}
		\varphi \sim \psi_{j,1} \circ \Phi_j \sim \Psi_{j,1} \circ \tilde{\xi}_{j,1} \circ \Phi_j \sim \Psi_{j,1} \circ \Xi_{j,1} \circ \Phi_j  = \Phi_{j+1},
	\end{align*}
	where the first approximation comes from $(4_j)$, the second by \eqref{equation: psi_jt}, and the third due to \eqref{equation: settingd1}. Combining this with \eqref{equation: settingc} we have
	\begin{align} \label{equation: lambda=one}
		Z_{b+2}^\RR  \ssubset \Phi_{j+1} \left( Z_{c}^\RR \right).
	\end{align}
	Furthermore consider the isotopy
	\[	
		\hat{\xi}_{j,t} = \psi_{j,t} \circ  \restr {(\Psi_{j,t} \circ \Xi_{j,t} )^{-1} }{X_\RR} 	
	\]
	and its time derivative $\hat{V}_t$. Moreover let $\lambda$ be a cutoff function on $X_\RR$ such that it is zero on $Z^\RR_{b}$ and equal to one outside $Z^\RR_{b+1}$. Finally let \textcolor{black}{$\psi_{j+1, t}$} denote the flow map of the vector field $\lambda \hat{V}_t$. 
    \medskip

    We check the conditions for the induction step: 
	
		($1_{j+1}$)	By the definition of $\Phi_{j+1}$ we have that
						\[
							\Phi_{j+1} (Z_{r_{j+1}}) = \Psi_{j,1} \circ \Xi_{j,1} \circ \Phi_j (Z_{r_{j+1}}).
						\]
						The claim follows from the fact that $\Xi_{j,t}$ approximates the identity on $Z_a$ and from the choices of $a, s_{j+1}$ in \eqref{equation: settinga}, \eqref{equation: Sj+1FromPsi} respectively. 
		%\smallskip
		
		($2_{j+1}$)	Here we want to estimate 
						\[
							 d_k( j^k \Phi_{j+1} (p) - j^k \Phi_{j}(p) )  =  d_k(j^k (\Psi_{j,1} \circ \Xi_{j,1} \circ \Phi_j)(p) - j^k \Phi_{j} (p)), \quad \forall p \in Z_{r_{j}}.
						\] 
						By $(1_j)$ we have 
						\[ 	
							\Phi_j(Z_{r_j}) \subset Z_{s_j}.		
						\] 
						The estimate follows from the fact that $\Psi_{j,1}$ and $ \Xi_{j,1}$ approximate the identity on $Z_{s_j}$. \iffalse For the inverses
						\[ 
							\| \Phi^{-1}_{j+1} - \Phi^{-1}_{j} \|_{C^k(Z_{S_{j}})} =  \|  \Phi^{-1}_j \circ \Xi^{-1}_{j,1} \circ \Psi^{-1}_{j,1} - \Phi^{-1}_{j} \|_{C^k(Z_{S_{j}})} 
						\]
						the claim follows by a similar argument. \fi
		%\smallskip
		
		%($3_{j+1}$)	A flow map at time zero is the identity.
		%\smallskip
		
		($3_{j+1}$)	Because $\psi_{j+1,t}$ is obtained by interpolating between the identity on $Z_b^\RR$ and $\hat{\xi}_{j,t}$ outside $Z_{b+2}^\RR$, it follows that $\psi_{j+1,t} (p) = (p)$ for $p$ near $Z^\RR_{s_{j+1}} $ because $b = s_{j+1} + 1$.
		%\smallskip
		
		($4_{j+1}$)	To see that 
						\[
							 d_k (j^k (\psi_{j+1, 1} \circ \Phi_{j+1})(p) - j^k \varphi (p) ) < \epsilon (p), \quad \forall p \in X_\RR,
						\]
						we consider separately: 
		%\smallskip
			\begin{enumerate}[label=(\roman*)]
				\item		On the complement of $(\Phi_{j+1})^{-1} \left( Z^\RR_{c} \right)$, we have $\lambda = 1$ at $\Phi_{j+1}(p)$ thanks to \eqref{equation: lambda=one}.
						Then $ \psi_{j+1, 1}$ is 
						\[	
						\hat{\xi}_{j,1} = \psi_{j,1} \circ ( \Psi_{j,1} \circ \Xi_{j,1} )^{-1}, 
						\]
						which together with $\Phi_{j+1} = \Psi_{j,1} \circ \Xi_{j,1} \circ \Phi_j$ reduces this case to $(4_j)$. 
				\item		On $\Phi_j^{-1} \left( Z_d^\RR \right)$, we have by the choice of $\Xi_{j,1}$ that 
						\begin{align}\label{equation: phi(j+1)}
						\Phi_{j+1} = \Psi_{j,1} \circ \Xi_{j, 1} \circ \Phi_j  \sim \Psi_{j,1} \circ \tilde{\xi}_{j, 1} \circ \Phi_j  \sim \psi_{j,1} \circ \Phi_j	,
						\end{align}
						which approximates $\varphi$ by the induction hypothesis $(4_j)$. Moreover, this implies that $\Psi_{j,1} \circ \Xi_{j, 1} \sim \psi_{j,1}$ on $Z_d^\RR$. Then on 
						\begin{align*} 
							\Phi_{j+1} \circ \Phi_j^{-1} \left( Z_d^\RR \right)  = \Psi_{j,1} \circ \Xi_{j, 1} \left( Z_d^\RR \right)	
						\end{align*}
						we have
						\begin{align*}
							\hat{\xi}_{j,1}  &= 	\psi_{j,1} \circ \left(  \Psi_{j,1} \circ \Xi_{j,1} \right)^{-1}  \\
							&\sim \psi_{j,1} \circ \left(\psi_{j,1} \right)^{-1}  =  \mathrm{id}	.
						\end{align*}
						Since 
						\[
							Z_{b+2}^\RR \subset \Phi_{j+1}(Z_c^\RR) \subset \Phi_{j+1} \circ \Phi_j^{-1} (Z_d^\RR) ,
						\]
						this justifies our interpolation $\psi_{j+1, t}$ between the identity on $Z_b^\RR$ and $\hat{\xi}_{j,t}$ outside $Z_{b+2}^\RR$. Hence $\psi_{j+1,1}$ is either close to the identity (reducing to the above \eqref{equation: phi(j+1)}) or equal to 
						\[	\hat{\xi}_{j,1} = \psi_{j,1} \circ  \restr {(\Psi_{j,1} \circ \Xi_{j,1} )^{-1} }{X_\RR} ,
						\]
						which composed with $\Phi_{j+1}$ again reduces it to $(4_j)$.
				\item		It suffices to consider the above two cases, because by the choice of $d$ in \eqref{equation: settingd2} and the approximation $ \psi_{j,1} \sim \Psi_{j,1} \circ \tilde{\xi}_{j,1} $ in \eqref{equation: psi_jt}
						\[	
							Z_{c}^\RR \ssubset \Psi_{j,1} \circ \Xi_{j,1} \left( Z_d^\RR \right) = \Phi_{j+1} \circ \Phi_j^{-1} \left( Z_d^\RR \right)	,	
						\]
						it follows
						\[	
							(\Phi_{j+1})^{-1} \left( Z_{c}^\RR \right)  \ssubset (\Phi_j)^{-1} \left( Z_d^\RR \right)	.
						\]
		\end{enumerate}
	This completes the induction step. 
    \bigskip 

    The desired $\Phi \in \saut(X)$ that should approximate $\varphi$ on $X_\R$ is the limit $\lim_{j \to \infty} \Phi_j$. Indeed the holomorphic automorphism $\Phi_{j+1} \circ \Phi_j^{-1} = \Psi_{j,1} \circ \Xi_{j,1}$ approximates the identity on $Z_{r_j}$. Since $r_j \ge j -1$, the sublevel set $Z_{r_j}$ exhausts $X$ in the limit. Additionally we have $r_{j+1} = s_j + 1 > r_j + 2$. Choose $\{ \varepsilon_j \}_{j \in \NN} \subset \RR_{>0}$ to have finite sum such that
	\[	0 < \varepsilon_{j+1} < \mathrm{dist} ( Z_{r_{j}} , X \setminus Z_{r_{j+1}} )
	\]
	where $\mathrm{dist}(\cdot, \cdot)$ is the distance function on $\CC^{2N}$ restricted to $X$. 
	Therefore the limit 
	\[	\lim_{j \to \infty} \Phi_{j+1} \circ \Phi_j^{-1}
	\]
	exists uniformly on compact subsets on 
	\[	
        \bigcup_{j = 1}^{\infty} \left( \Phi_{j+1} \circ \Phi_j^{-1} \right)^{-1}  (Z_{r_j}) = X
	\]
	(see e.g.\@ Forstneri{\v c} \cite{MR1760722}*{Proposition 5.1}). Thus
	\[	 \lim_{j \to \infty} \Phi_{j} = (\lim_{j \to \infty} \Phi_{j+1} \circ \Phi_j^{-1} ) \circ \Phi_1
	\]
	converges to a holomorphic automorphism $\Phi$ in $\saut(X)$. 
	Moreover, $(1_j)$ and $(3_j)$ guarantee that $\psi_{j, 1} \circ \Phi_j$ is $\Phi_j$ on $Z_{r_j}^\R$. Then $(4_j)$ says $\Phi_j$ approximates $\varphi$ on $Z_{r_j}^\R$ and thus in the limit $\Phi$ approximates $\varphi$ on the entire $X_\R$.
\end{proof}

%\newpage

\subsection{The Group of Volume-Preserving Diffeomorphisms}
The goal of this subsection is to prove Theorem \ref{Global-Carleman-vol} which is Theorem \ref{Carleman-vol} in the introduction. The volume-preserving vector fields do not form a module over the functions, thus approximation of volume-preserving vector fields does not directly follow from approximation on the level of functions.  The standard  technique to overcome this difficulty is cohomological considerations as for example  in our next result.

\begin{lemma} \label{n-2-form}
    Let $M$ be a smooth manifold of dimension $n$ with $H^{n-1}(M)=0$ 
and let $\rho \colon M \to \R_{\ge 0}$ be a smooth exhaustion function whose critical points all lie in $B_r:=\{ \rho < r \}$ for some $r >0$. 

(i) Then for every closed $(n-1)$-form $\beta$ vanishing on $B_r$ and every $r' < r$ there exists an $(n-2)$-form $\alpha$ with $d \alpha = \beta$ vanishing on $B_{r'}$. 

(ii) Moreover, for every closed $(n-1)$-form $\beta$ and every $r' < r$ there exists an $(n-2)$-form $\alpha$ with $d \alpha = \beta$ such that $\| \alpha \|_{C^k(B_{r'})} \le C \| \beta \|_{C^k(B_r)}$ for some constant $C$ depending on $k$ and $r$.
\end{lemma}

\begin{proof}
    (i) Choose any smooth Riemannian metric on $M$ and denote by $\mathrm{grad} \rho$ the gradient of $\rho$ with respect to this metric. Let $\eta \colon M \to [0,1]$ be a smooth cutoff function with $\eta$ vanishing on $\overline{B}_{r'}$ and identically one on $\{ \rho > r \}$. Denote by $\varphi \colon [0, \infty) \times M \to M$ the flow of the vector field 
    \begin{align*}
         -\eta(x) \frac{\mathrm{grad}\rho(x)}{\| \mathrm{grad}\rho(x) \|^2}.
    \end{align*}
    The smooth map $\psi \colon M \to M$ defined by $\psi(x) = \varphi_{\rho(x) - r'}(x)$ is the identity on $\overline{B}_{r'}$ and has image contained in $B_r$. 
    
    By $H^{n-1}(M)=0$ there exists an $(n-2)$-form $\alpha_1$ with $d \alpha_1 = \beta$. The desired primitive $\alpha$ is given by $\alpha_1 - \psi^* \left(\restr{\alpha_1}{B_r}\right)$. 
    \smallskip 

    (ii) Since $H^{n-1}(M)=0$ and $\beta$ is closed, the Hodge-Morrey decomposition of differential forms on compact manifolds with boundary (see e.g.\@ Schwarz \cite{MR1367287}*{Theorem 2.4.2})
    gives us an $(n-2)$-form $a$ on $\overline{B}_r$ with $d a = \restr{\beta}{\overline{B}_r}$. By \cite{MR1367287}*{Lemma 2.4.11} the primitive $a$ can be chosen such that the desired estimate is valid with respect to Sobolev norms on $\overline{B}_r$: 
    \begin{align*}
        \| a \|_{W^{k+1,p}(\overline{B}_r)} \le C \| \beta \|_{W^{k,p}(\overline{B}_r)} \text{ for } 1 < p < \infty. 
    \end{align*}
    Together with the Sobolev-Morrey embedding for $p > n$
    % \[
    %     \| a \|_{C^{k}(\overline{B}_r)} \le \| a \|_{C^{k,1- \frac{n}{p}}(\overline{B}_r)} \le C \| a \|_{W^{k+1,p}(\overline{B}_r)}
    % \]
    and the compactness of $\overline{B}_r$ we get the desired estimate for the local forms
    \[
        \| a \|_{C^{k}(\overline{B}_r)} \le C \| a \|_{W^{k+1,p}(\overline{B}_r)} \le C \| \beta \|_{W^{k,p}(\overline{B}_r)}  \le C \| \beta \|_{C^{k}(\overline{B}_r)}.
    \]
    To obtain the global primitive, we take a cutoff function $\chi$ which is 1 on $B_{s}$ with $r'<s<r$ and 0 outside $B_r$. Then $\beta - d(\chi a)$ is closed and vanishes on $B_s$ and part (i) yields an $(n-2)$-form $\tilde{\alpha}$ vanishing on $B_{r'}$ such that $d \tilde{\alpha}=\beta - d(\chi a)$. The global form $\alpha:= \tilde{\alpha} + \chi a$ has the desired properties.
\end{proof}

\begin{theorem} \label{theorem: LocalCarleman-vol}
	Let $X_\R$ be a smooth manifold endowed with a real volume form $\Omega_\R$ with $H^{n-1}(X_\R)=0$ and $(X, \Omega, \sigma)$ a  Stein volume complexification of $(X_\R, \Omega_\R)$. Assume there exist a real holomorphic embedding $X \hookrightarrow \C^q$ and $r \ge 0$ such that the restriction of $\rho$ to $X_\R$ has only critical points in the interior of $Z^\R_r$ (from Definition \ref{sublevelset}).
    Let $\varphi \colon [0,1] \times X_\RR \to X_\RR$ be a smooth isotopy of diffeomorphisms in $\diff_{\Omega_\R}(X_\R)$  such that $\varphi_0 = \id$ and $\restr{\varphi_t}{Z_r^\RR} = \id_{Z_r^\RR}$ for all $t$ in $[0,1]$. 
    
    If  $X$ has the $\sigma$-volume density property, then for any $k \in \N$, any $\epsilon>0$, any $b \in (0,r)$ and any $a > b$, there exists a smooth isotopy of automorphisms $\Phi \colon [0,1] \times X \to X$ such that $\Phi_{t} \in \saut_\Omega(X)$ and 
	\begin{align*}
		  d_k( j^k \Phi_{t}(p) - j^k \varphi_t(p)) &< \epsilon \text{ for all } p \in Z_{a}^\R, 	\\
		d_k( j^k \Phi_{t} (p) - j^k \mathrm{id} (p) )		&< \epsilon  \text{ for all } p \in Z_{b}, \quad \text{ for all } t \in [0,1].	 %\label{ApproxOnZb}
	\end{align*}
\end{theorem}
\begin{proof}
    We introduce an auxiliary time variable $s$
    \[
        \psi \colon [0,1] \times [0,1] \times X_\R \to X_\R, (s, t, p) \mapsto \varphi(st, p)
    \]
    and consider $\varphi_t$ as the time-1 map of the time-$s$ dependent volume-preserving vector field 
    \[
        V \colon [0,1] \times [0,1] \times X_\R \to T X_\R 
    \]
    defined by 
	\begin{align*} %\label{equation: V_t}
		\frac{d } {d s} \varphi_{st} (p) = V (s,t, \varphi_{st}(p)) = V_{s,t}(\varphi_{st}(p)) , \quad s, t \in [0,1], \, p \in X_\R,		
	\end{align*}
    with the initial condition $\psi(0,t,p) = \varphi(0,p) = p$. 
    Since $\restr{\varphi_t}{Z_r^\RR} = \id_{Z_r^\RR}$ for $t \in [0,1]$, we can extend $V_{s,t}$ to be 0 in an open neighborhood of $Z_{b'}$ for $s, t \in [0,1]$, where $b' \in (b,r)$. 
    Choose $a'>0$ such that 
    \begin{align*}
        \bigcup_{ t \in [0,1] }  \varphi_t ( Z_{a}^\RR ) \subset Z_{a'}^\RR. 
    \end{align*}
    Since $H^{n-1}(X_\R) = 0$, there exists a smooth family of $(n-2)$-forms $\alpha_{s,t}$ on $X_\R$ such that $i_{V_{s,t}} \Omega_\R = d \alpha_{s,t}$. 
    By Lemma \ref{n-2-form} we can assume there is a $r' \in (b', r)$ such that $\alpha_{s,t}$ is zero on $Z_{r'}^\R$.

    Let $Q$ be a compact Hausdorff space whose points we denote by $q$. Given a continuously $q$-dependent function $h_q$ on $S= Z^\R_{a'} \cup Z_{b'}$ which is in $C^{k+1}(S) \cap \holo(Z_{b'})$, by the holomorphic convexity of $S$ in $X$ (Lemma \ref{Saturn-holoconv}) and the parametric version of the Mergelyan theorem by Forn{\ae}ss--Forstneri{\v c}--Wold \cite{MR4264040}*{Theorem 20}, $h_q$ is approximable in the $C^{k+1}$-norm by a continuously $q$-dependent function $P_q$ which is holomorphic in a neighborhood $U_S$ of $S$. 
    The same is true for sections in the holomorphic vector  bundle $\wedge^{n-2} T^* X$ since any holomorphic vector bundle on a Stein space is stably trivial, for details see e.g.\@ Forstneri{\v c} \cite{MR2975791}*{p.\@ 64} and Giraldo--S{\'a}nchez-Arellano \cite{MR4904178}*{Theorem 4.9}. Hence $\alpha_{s,t}$ can be approximated in the $C^{k+1}$-norm on $S$ by a continuously $(s,t)$-dependent holomorphic $(n-2)$-form $A_{s,t}$ for $s,t \in [0,1]$.
    Moreover, we can assume that $A(s,t,p)$ is a finite sum of the form $\sum_{n,m} s^n t^m A_{nm}(p)$ by a patching argument as in the proof of \cite{MR1314745}*{Lemma 1.2}. 

    Then by 
    the Cartan--Oka--Weil theorem, we can approximate $A_{nm}$ on a relative compact Runge open neighborhood of $S$ by a holomorphic $(n-2)$-form $\widetilde{A}_{nm}$ on $X$. Since both $A_{nm}$ and $\widetilde{A}_{nm}$ are holomorphic in a neighborhood of $S$, the approximation holds in the $C^{k+1}$-norm.

    Since $\alpha_{s,t}$ is a real $(n-2)$-form on $X_\R$, we can replace $\widetilde{A}_{nm}$ by $(\widetilde{A}_{nm} + \overline{\sigma_* \widetilde{A}_{nm}})/2$ such that 
    \[
        \widetilde{A}_{s,t}(p) = \widetilde{A}(s,t,p) = \sum_{n,m} s^n t^m \widetilde{A}_{nm}(p)
    \]
    is real holomorphic with respect to $\sigma$ and approximates $\alpha_{s,t}$. Denote by $\widetilde{W}_{nm}$ the real holomorphic volume-preserving vector field corresponding to $\widetilde{A}_{nm}$, namely $i_{\widetilde{W}_{nm}} \Omega = d \widetilde{A}_{nm}$. 

    By the $\sigma$-density property and Remark \ref{replace-Lie-by-Sum}, $\widetilde{W}_{nm}$ can be approximated on $U_S$ by a complex linear combination $L$ of complete real holomorphic volume-preserving vector fields. 
    Since the real flows of real holomorphic vector fields preserve $X_\R$, we may assume that the coefficients in this linear combination are real. Indeed we can separate the coefficients of $L$ and thus also $L$ into real and imaginary parts $\widetilde{W}_{nm} \approx L_1 + i L_2$. Acting on both sides with $\overline{\sigma_*(\cdot)}$, we get $\widetilde{W}_{nm} \approx L_1 - i L_2$ since $\widetilde{W}_{nm}$ and $L_1, L_2$ are real holomorphic. Thus $L_2 \rvert_{U_S} \approx 0$ and $L_1$ approximates $\widetilde{W}_{nm}$ on $U_S$. Hence $\widetilde{W}_{s,t}$ can be approximated on $U_S$ (thus $V_{s,t}$ on $S$) by a finite sum $\sum_{l=1}^N p_l(s,t) \widetilde{W}_l$ where $p_l(s,t)$ is a polynomial and $\widetilde{W}_l$ is a complete real holomorphic volume-preserving vector field on $X$.

    By definition for all $t \in [0,1]$ the vector field  $V_{s,t}$ is integrable for time $s \in [0,1]$ for all starting points $x \in Z_a^\R \cup Z_b$ and its flow is contained in $S = Z^\R_{a'} \cup Z_{b'} \subset U_S$, thus there is an open neighborhood $U$ containing $Z_a^\R \cup Z_b$ such that $\sum_{l=1}^N p_l(s,t) \widetilde{W}_l$ is integrable for time $s \in [0,1]$ for all initial points $x$ in $U$. 

    For a fixed $n \in \NN$, $j \in \{0, 1, \dots, n-1\}$, we denote by $\Phi^{(s, l)}_\theta$ the time-$\theta$ flow map of the time-$s$ independent vector field $ p_l(s,t) \widetilde{W}_l$, $l \in \{ 1, \dots, N\}$ and consider the composition
    \begin{align*} %\label{algorithm}
        \Phi^{(s)}_{\theta} =  \Phi^{(s, 1)}_{\theta} \circ \Phi^{(s, 2)}_{\theta} \circ \cdots \circ \Phi^{(s,N)}_{\theta} .
    \end{align*}
    Now we take 
    \begin{align} \label{composition-vol}
		 \Phi^{(\frac{n-1}{n})}_{1/n} \circ  \Phi^{(\frac{n-2}{n})}_{1/n}  \circ \cdots \circ  \Phi^{(\frac{1}{n})}_{1/n} \circ  \Phi^{(0)}_{1/n} 
	\end{align}
    which approximates the time-$1$ flow of $\sum_{l=1}^N p_l(s,t) \widetilde{W}_l$ uniformly on a relatively compact open neighborhood of any compact subset in $U$ for large enough $n$ \cite{MR0515141}*{Theorem 2.1.26}. Indeed an algorithm for the vector field $\left( 1, \sum_{l=1}^N p_l(s,t) \widetilde{W}_l \right) $ on $[0,1] \times ([0,1]\times X)$ is given by 
    \begin{align*}
        [0, \epsilon) \times [0,1]^2 \times X \to  [0,1]^2 \times X, \quad (\theta, s, t, p) \mapsto (s + \theta, t, \Phi^{(s)}_\theta(p))  
    \end{align*}
    whose $n$th iterate after projecting to $X$ is \eqref{composition-vol}. 
    By the Cauchy estimate, the approximation is in the $C^k$-norm and \eqref{composition-vol} depends smoothly on $t$. 
    
    Finally, since approximation of $V_{s,t}$ by the vector field $\sum_{l=1}^N p_l(s,t) \widetilde{W}_l$ implies the approximation of $\varphi_t$ by the time-1 flow of $\sum_{l=1}^N p_l(s,t) \widetilde{W}_l$, it follows that \eqref{composition-vol} approximates $\varphi_t$ on a relatively compact open neighborhood of $Z^\R_a \cap Z_b$. 
\end{proof}

\begin{remark} \label{Local-Carleman-w/o-interpolation-vol}
    The same proof also shows that for any $a \ge 0$ and any smooth isotopy of diffeomorphisms $\varphi \colon [0,1] \times X_\RR \to X_\RR$ in $\diff_{\Omega_\R}(X_\R)$, the cohomological condition $H^{n-1}(X_\R)=0$ and the $\sigma$-volume density property of $X$ imply the existence of a smooth isotopy of automorphisms $\Phi \colon [0,1] \times X \to X$ such that $\Phi_{t} \in \saut_\Omega(X)$ and $d_k( j^k \Phi_{t}(p) - j^k \varphi_t(p)) < \epsilon$ for all $p \in Z_{a}^\R$ and $t \in [0,1]$.
\end{remark}

Next comes the global approximation theorem for volume-preserving diffeomorphisms. It is based on a similar induction scheme as for diffeomorphisms and the main variation is to glue $(n-2)$-forms in place of vector fields.

\begin{theorem} \label{Global-Carleman-vol}
    Let $X_\R$ be a smooth manifold endowed with a real volume form $\Omega_\R$ with $H^{n-1}(X_\R)=0$ and $(X, \Omega, \sigma)$ a Stein volume complexification of $(X_\R, \Omega_\R)$. Assume there exist a real holomorphic embedding $X \hookrightarrow \C^q$ and $r \ge 0$ such that the restriction of $\rho$ to $X_\R$ has only critical points in the interior of $Z^\R_r$ (from Definition \ref{sublevelset}).
    
    If $X$ has the $\sigma$-volume density property, then for every volume-preserving diffeomorphism $\varphi \in \diff_{\Omega_\R,0}(X_\R)$ smoothly isotopic to $\id$, for any $k \in \N$ and any positive continuous function $\epsilon$ on $X_\R$, there exists a holomorphic automorphism $\Phi \in \saut_\Omega(X)$ such that
	\[	
        d_k( j^k \Phi (p) - j^k \varphi(p) ) < \epsilon (p) \quad \text{ for all }\, p \in X_\R.	
    \]
\end{theorem}

\begin{proof}
    By assumption there exists a smooth isotopy $\varphi_t$ in $\diff_{\Omega_\R}(X_\R), t \in [0,1]$ with $\varphi_0 = \mathrm{id}$ and $\varphi_1= \varphi$. 
	%\smallskip
    
	Given any natural number $j$ and any positive $\varepsilon_j$, assume that we have a holomorphic automorphism $\Phi_j \in \saut_\Omega(X)$, a smooth isotopy $\psi_{j,t}$ of volume-preserving diffeomorphisms of $X_\R$ with $\psi_{j,0}=\id$, real numbers $r_j, s_j$ with $r_j \ge j-1, s_j \ge r_j + 1$ such that
	\smallskip
	\begin{enumerate}
		\item[($1_j$)] 	The image of $ Z_{r_j} $ under $\Phi_j$ is contained in $Z_{s_j}$. 
		\item[($2_j$)]	For $j \ge 2$ 
					\[	 d_k (j^k \Phi_j (p) - j^k \Phi_{j-1}(p) ) < \varepsilon_j \quad  \forall \, p \in Z_{r_{j-1}}.
                    \] %, \quad  \| \Phi^{-1}_j - \Phi^{-1}_{j-1} \|_{C^k(Z_{s_{j-1}})} < \varepsilon_j		\]
		\item[($3_j$)]	On $Z^\R_{s_j }$ we have $\psi_{j,t}$ is the identity for $t$ in $[0,1]$. %\textcolor{red}{why $\varepsilon_j$?}
		\item[($4_j$)]	For all $p$ in $X_\R$
					\[	d_k (j^k (\psi_{j, 1} \circ \Phi_j) (p) - j^k \varphi (p) ) < \epsilon (p).
                    \]
	\end{enumerate}

    \textbf{Induction base}: Choose $p_0 \in \R^q \setminus X_\R$ so that $\mathrm{dist}_{\C^q}(p_0, X_\R) < \mathrm{dist}_{\C^q}(p_0, X \setminus X_\R)$. For $j = 1$ take \textcolor{black}{$r_1 = \mathrm{dist}_{\C^q}(p_0, X_\R)$}. Then $Z_{r_1} = Z_{r_1}^\R$. 
	Choose \textcolor{black}{$s_1$} so that 
	\begin{align} \label{equation: S_1-vol}
		\varphi_1 \left( Z_{r_1+1}^\RR \right) \ssubset Z_{s_1}^\RR, \quad s_1 \ge r .
	\end{align}
	Choose $T > s_1 $. Then by \eqref{equation: S_1-vol}
	\begin{align}
		Z_{r_1}^\RR \subset Z_{r_1+1}^\RR  \ssubset (\varphi_1)^{-1} \left( Z_{s_1}^\RR \right) \subset (\varphi_1)^{-1} \left( Z_{T}^\RR \right).	\label{equation: R_1-vol}	
	\end{align}
	Using Remark \ref{Local-Carleman-w/o-interpolation-vol} we get a smooth isotopy $A_t \in \saut_\Omega(X)$ approximating $\varphi_t$ on the compact subset 
	\[
		\bigcup_{ t \in [0,1]} (\varphi_t)^{-1} \left( Z_{T + 3 }^\RR \right)
	\]
	of $X_\RR$. Since $A_t$ approximates $\varphi_t$ on this compact, we have 
	\begin{align}
		(A_t)^{-1} \left( Z_{T + 2}^\RR \right) 	\ssubset  	(\varphi_t)^{-1} \left( Z_{T + 3}^\RR \right).  \label{equation: r_3-vol}
	\end{align}
	In particular, $A_1$ approximates $\varphi_1$ on $A_1^{-1} \left( Z_{T + 2}^\RR \right) $, hence $\varphi_t \circ \restr{A_t^{-1}}{X_\RR}$ is close to the identity on $Z_{T + 2}^\RR$. Choose \textcolor{black}{$\Phi_1 = A_1$}.
	
	%\smallskip
	To construct an isotopy $\psi_{1,t}$ we interpolate between the identity on $Z_{T}^\RR $ and $\varphi_t \circ \restr{A_t^{-1}}{X_\RR}$ outside $Z_{T + 2}^\RR$. More precisely, let $W_t$ be the infinitesimal generator of the isotopy $\varphi_t \circ \restr{A_t^{-1}}{X_\RR}$. By $H^{n-1}(X_\R) = 0$ there exists an $(n-2)$-form $\alpha_t$ such that $i_{W_t} \Omega_\R = d \alpha_t$. Using $T > s_1 > r$ and Lemma \ref{n-2-form} we can choose $\alpha_t$ with $\restr{\alpha_t}{Z_{T+1}^\RR}$ close to zero. 
    Let $\gamma$ be a cutoff function which is zero on $Z_{T}^\RR$ and one outside $ Z_{T + 1}^\RR $.
	Take \textcolor{black}{$\psi_{1,t}$} to be the time-$t$ flow of the volume-preserving vector field corresponding to the $(n-1)$-form $d (\gamma \alpha_t)$. 
    \medskip 
	
	By the choice of $T$ we have that $Z_{s_1}^\RR$ is contained in $Z_{T}^\RR$, which implies that $\psi_{1,t}$ is the identity on $Z_{s_1}^\RR$ for all $t$ in $[0,1]$. This shows $(3_1)$. 
	\smallskip 
	
	To see that $(4_1)$ is satisfied, let $p \in X_\RR$ and consider separately
	\begin{enumerate}[label=(\roman*)]
		\item		$p \in A_1^{-1} \left( Z_{T}^\RR \right)$: $\psi_{1,1}$ is the identity at $A_1 (p)$ and $A_1$ approximates $\varphi_1$ by \eqref{equation: r_3-vol}. 
		\smallskip
		\item		$p \notin A_1^{-1} \left( Z_{T + 2}^\RR \right)$: $\psi_{1,1}$ is equal to $\varphi_1 \circ A_1^{-1}$ at $A_1 (p)$ by the choice of $\gamma$.
		\smallskip
		\item		$p \in A_1^{-1} \left( Z_{T + 2}^\RR \setminus Z_{T}^\RR \right)$: $A_1$ approximates $\varphi_1$ by \eqref{equation: r_3-vol} and $\psi_{1,1} \circ \varphi_1$ is the interpolation between $\mathrm{id} \circ \varphi_1$ 				and $\varphi_1 \circ A_1^{-1} \circ \varphi_1$. Here $\varphi_1 \sim A_1$ implies $A_1^{-1} \circ \varphi_1 \sim \mathrm{id}$.  
	\end{enumerate}
	
	Last, let us check that $(1_1)$ holds. By the choice of $A_t$, we may assume that there exists a small positive $\delta$ which is less than one, such that 
	\[
		\Phi_1 \left( Z^\RR_{r_1} \right) \subset 	\varphi_1 \left( Z_{r_1+ \delta}^\RR \right) \subset \varphi_1 \left( Z_{r_1+ 1}^\RR \right) \ssubset Z_{s_1}^\RR .
	\]
	The first inclusion follows from the fact that $\Phi_1 = A_1$ approximates $\varphi_1$ by \eqref{equation: R_1-vol} and the last inclusion is due to the choice of $s_1$ in $\eqref{equation: S_1}$. This concludes the induction base. 
    \bigskip

    \textbf{Induction step}: Take \textcolor{black}{$r_{j+1} = s_j + 1$} and choose $a > \max \{ r_j + 1, r_{j+1} \}$ such that
	\begin{align*}
		%\bigcup_{ t \in [0,1]} \psi_{j,t} (Z^\RR_{R_{j+1}}) &\ssubset Z^\RR_{a} 	\nonumber	\\
		%\bigcup_{ t \in [0,1]}  (\psi_{j,t})^{-1}( Z^\RR_{R_{j+1}}) &\ssubset  Z^\RR_{a} 	\nonumber \\
		\Phi_j ( Z_{r_{j+1}} ) \ssubset	Z_a. %\label{equation: settinga-vol}
	\end{align*}	
	Since $s_j \ge r$, by Theorem \ref{theorem: LocalCarleman-vol} there exists a smooth isotopy $\Psi_{j,t}$ in $\saut_\Omega(X)$, which approximates the identity near $Z_{s_j}$ and $\psi_{j,t}$ near $Z^\RR_{a + 2}$.  % Using $R = S_j + \varepsilon_j$, $a = a +2$, $b=S_j$
	Thus %the symplectic diffeomorphism
	\[	
		\xi_{j,t} = ( \restr{\Psi_{j,t}}{X_\RR} )^{-1} \circ \psi_{j,t}	
	\]
	approximates the identity on $Z_{a+2}^\RR$. 
	
	%\smallskip
	Moreover take a cutoff function $\chi$ on $X_\RR$ such that it is zero on $Z^\RR_a$ and equal to one outside $Z_{a+1}^\RR$. Let $V_t$ be the infinitesimal generator of the smooth isotopy $\xi_{j,t}$. By $H^{n-1}(X_\R)= 0$ there is an $(n-2)$-form $\beta_t$ such that $i_{V_t}\Omega_\R = d \beta_t$. Again by Lemma \ref{n-2-form} the primitive $\beta_t$ can be chosen to be close to zero on $Z^\R_{a+1}$. 
    Let $\tilde{\xi}_{j,t}$ be the flow map of the volume-preserving vector field corresponding to the $(n-1)$-form $ d (\chi \beta_t)$. Then $\tilde{\xi}_{j,t}$ is the identity on $Z^\RR_a$, close to the identity on $Z^\RR_{a+2}$, and equal to $\xi_{j,t}$ outside $Z^\RR_{a+2}$. Indeed on $Z^\R_{a+2} \setminus Z^\R_a$ we have $d(\chi \beta_t) = \chi d \beta_t + d \chi \wedge \beta_t$ close to zero by Lemma \ref{n-2-form}.

\medskip

	%\smallskip
	Then we have on $X_\RR$
	\begin{align}
		\Psi_{j,t} \circ \tilde{\xi}_{j,t}  &\sim  \psi_{j,t},  \label{equation: psi_jt-vol} \\ 
		(\Psi_{j,t} \circ \tilde{\xi}_{j,t})^{-1} &\sim (\psi_{j,t})^{-1},  \nonumber
	\end{align}	
	by the choices of $\Psi_{j,t}, \xi_{j,t}, \tilde{\xi}_{j,t}$.
	Next, choose \textcolor{black}{$s_{j+1} > a $} so that 
	\begin{align*} %\label{equation: Sj+1FromPsi-vol}
		\Psi_{j,1} ( Z_a ) \ssubset Z_{s_{j+1}}. 	
	\end{align*}
	Set $b = s_{j+1} + 1$ and pick $c$ and $d$ so that
	\begin{align} \label{equation: settingc-vol}
		Z_{b+2}^\RR &\ssubset \varphi \left( Z_c^\RR \right), \\
		\Phi_j   \left(  Z_c^\RR \right) &\ssubset Z_d^\RR, \label{equation: settingd1-vol} \\
		Z_{c}^\RR &\ssubset \psi_{j,1} \left( Z_d^\RR \right). \nonumber %\label{equation: settingd2}
	\end{align}
	
	Next, apply Theorem \ref{theorem: LocalCarleman-vol} to obtain an isotopy $\Xi_{j,t} \in \saut_\Omega(X)$, which approximate $\tilde{\xi}_{j,t}$ on $Z^\RR_{d}$ and the identity on $Z_a$. Set \textcolor{black}{$\Phi_{j+1} = \Psi_{j,1} \circ \Xi_{j,1} \circ \Phi_j$}. The above choices of $c$ and $d$ allow us to approximate on $Z_c^\RR$
	\begin{align*}
		\varphi \sim \psi_{j,1} \circ \Phi_j \sim \Psi_{j,1} \circ \tilde{\xi}_{j,1} \circ \Phi_j \sim \Psi_{j,1} \circ \Xi_{j,1} \circ \Phi_j  = \Phi_{j+1},
	\end{align*}
	where the first approximation comes from $(4_j)$, the second by \eqref{equation: psi_jt-vol}, and the third due to \eqref{equation: settingd1-vol}. Combining this with \eqref{equation: settingc-vol} we have
	\begin{align*} %\label{equation: lambda=one-vol}
		Z_{b+2}^\RR  \ssubset \Phi_{j+1} \left( Z_{c}^\RR \right).
	\end{align*}
	Furthermore consider the isotopy
	\[	
		\hat{\xi}_{j,t} = \psi_{j,t} \circ  \restr {(\Psi_{j,t} \circ \Xi_{j,t} )^{-1} }{X_\RR} 	
	\]
	and its time derivative $\hat{V}_t$. Let $\hat{\beta}_t$ be an $(n-2)$-form such that $i_{\hat{V}_t} \Omega_\R = d \hat{\beta}_t$ and $\hat{\beta}_t$ is close to zero on $Z^\R_{b+1}$. Moreover let $\lambda$ be a cutoff function on $X_\RR$ such that it is zero on $Z^\RR_{b}$ and equal to one outside $Z^\RR_{b+1}$. Finally let \textcolor{black}{$\psi_{j+1, t}$} denote the flow map of the volume-preserving vector field corresponding to $d(\lambda \hat{\beta}_t)$. 
    \medskip

    The rest of the proof is identical to the proof of Theorem \ref{Global-Carleman}.
\end{proof}

\subsection{The Group of Hamiltonian Diffeomorphisms} \label{sec: Ham-approximation}
The aim of this subsection is to prove Theorem \ref{Global-Carleman-ham} which is Theorem \ref{Carleman-ham} in the Introduction. The proof is similar to Theorem \ref{Global-Carleman} and we indicate only the differences. 

Let $(X, \omega)$ be a smooth symplectic manifold. A {symplectic isotopy} is a jointly smooth map $\phi \colon [0, 1] \times X \to X$ such that $\phi_t$ is symplectic for every $t$ in $[0,1]$ and $\phi_0 = \id$. The isotopy $\phi_t$ is generated by a smooth family of vector fields $V_t$ with $d \phi_t / dt= V_t \circ \phi_t, \, \phi_0 = \id$. Since $\phi_t$ is symplectic, the vector field $V_t$ is symplectic and the 1-form $i_{V_t} \omega$ is closed for $t \in [0,1]$.  A {\bf Hamiltonian isotopy} is a symplectic isotopy such that the closed 1-form $i_{V_t} \omega$ is exact, namely $V_t$ is Hamiltonian, for all $t \in [0,1]$. In this case there exists a smooth family of Hamiltonian functions $H_t$ such that $i_{V_t}\omega = d H_t$.
We call a symplectomorphism which is the endpoint of a Hamiltonian isotopy a \textbf{Hamiltonian diffeomorphism}. Hamiltonian diffeomorphisms form a normal subgroup of the group of symplectic diffeomorphisms, see Dusa McDuff and Dietmar Salamon \cite{MR3674984}*{\S 3.1}.

\begin{theorem} \label{theorem: LocalCarleman-ham}
	Let $X_\R$ be a smooth manifold endowed with a real symplectic form $\omega_\R$ and $(X, \omega, \sigma)$ a symplectic complexification of $(X_\R, \omega_\R)$ which is Stein. Let $\varphi \colon [0,1] \times X_\RR \to X_\RR$ be a Hamiltonian isotopy and $r \ge 0$ such that $\restr{\varphi_t}{Z_r^\RR} = \id_{Z_r^\RR}$ for all $t$ in $[0,1]$. 
    
    If $X$ has the $\sigma$-Hamiltonian density property, then for any $k \in \N$, any $\epsilon>0$, any $b \in (0,r)$ and any $a > b$, there exists a Hamiltonian isotopy of holomorphic automorphisms $\Phi \colon [0,1] \times X \to X$ such that $\Phi_{t} \in \saut_\omega(X)$ and 
	\begin{align*}
		  d_k( j^k \Phi_{t}(p) - j^k \varphi_t(p)) &< \epsilon \text{ for all } p \in Z_{a}^\R, 	\\
		d_k( j^k \Phi_{t} (p) - j^k \mathrm{id} (p) )		&< \epsilon  \text{ for all } p \in Z_{b}, \quad \text{ for all } t \in [0,1].	 %\label{ApproxOnZb}
	\end{align*}
\end{theorem}

\begin{proof}
    The proof is analogous to the proof of Theorem \ref{theorem: LocalCarleman}. In the places where we approximated vector fields we approximate Hamiltonian functions. For more details see \cite{Cotan}*{Theorem 5.18}.
\end{proof}

\begin{theorem} \label{Global-Carleman-ham}
    Let $X_\R$ be a smooth manifold endowed with a real symplectic form $\omega_\R$ and $(X, \omega, \sigma)$ a symplectic complexification of $(X_\R, \omega_\R)$ which is Stein. If $X$ has the $\sigma$-Hamiltonian density property, then for every Hamiltonian diffeomorphism $\varphi \in \ham_{\omega_\R}(X_\R)$, for any $k \in \N$ and any positive continuous function $\epsilon$ on $X_\R$, there exists a holomorphic automorphism $\Phi \in \saut_\omega(X)$ such that
	\[	
        d_k( j^k \Phi (p) - j^k \varphi(p) ) < \epsilon (p) \quad \text{ for all }\, p \in X_\R.	
    \]
\end{theorem}
\begin{proof}
    The proof is analogous to the proof of Theorem \ref{Global-Carleman}. In the three places where we modified a real vector field by multiplying it with a cutoff function we modify a Hamiltonian vector field by multiplying its Hamiltonian function with a cutoff function.  This ensures that the modified vector field is again Hamiltonian.
\end{proof}

\section{Applications}
In this section we find concrete examples where the assumptions of the approximation theorems \ref{Global-Carleman}, \ref{Global-Carleman-vol} and \ref{Global-Carleman-ham} are fulfilled. In addition we give criteria for the $\sigma$-density property and the $\sigma$-volume density property, Theorem \ref{sigma-DP} and Theorem \ref{sig-VDP}, based on similar known criteria for density properties developed by Kaliman and Kutzschebauch in \cites{MR2385667, MR3492044}. It is clear that many more examples can be obtained in a similar manner.

\subsection{Approximation for the Diffeomorphism Group of  Real Lie Groups} \

\begin{definition} \label{flex-compatiblePair}
    (i) Let $X$ be a Stein manifold and $\sigma \colon X \to X$ an antiholomorphic involution. 
    We say that $X$ is $\sigma$-\textbf{holomorphically} \textbf{flexible}, if complete real holomorphic  vector fields on $X$ span the tangent space $T_p X$ at every point $p \in X$.

    If $X$ is equipped with a holomorphic volume form, we say $X$ is $\sigma$-\textbf{holomorphically volume} \textbf{flexible}, if complete volume-preserving real holomorphic  vector fields on $X$ span the tangent space $T_p X$ at every point $p \in X$.

    (ii)  A pair $(V,W)$ of complete holomorphic vector fields is a \textbf{semicompatible pair} if $\overline{ \mathrm{Span}} (\ker V \cdot \ker W)$ contains a nontrivial ideal $I \subset \holo(X)$. Here the closure is taken in the compact-open topology.
    
    (iii) A semicompatible pair $(V,W)$ is a \textbf{compatible pair} if either (1) there exists $h \in \ker W$ such that $V(h) \in \ker V \setminus \{0\}$, or (2) there exists $h$ such that $V(h) \in \ker V \setminus \{0\}$ and $W(h) \in \ker W \setminus \{0\}$. 
\end{definition}

\begin{theorem} \label{sigma-DP}
    Let $X$ be a Stein manifold and $\sigma \colon X \to X$ be an antiholomorphic involution such that the fixed-point set $X_\R=\mathrm{Fix}(\sigma) \neq \emptyset$. If $X$ 
    \begin{enumerate}
        \item admits a compatible pair of complete real holomorphic vector fields $(V,W)$, and
        \item is $\sigma$-holomorphically flexible, and
        \item there is a point $p \in X$ such that the orbit of $W(p)$ under the induced action of the isotropy group $(\saut(X))_p$ on $T_p X$ contains a basis,
    \end{enumerate}
    then $X$ has the $\sigma$-density property. 
\end{theorem}
\begin{proof}
    By definition of $\sigma$-density property we can take complex Lie combinations of complete real holomorphic vector fields. Denote this complex Lie algebra by $\lie^\sigma_X$. 

    For a real complete field $V$, if $f \in \ker V$ or $f \in \ker V^2 \setminus \ker V$ then also $f V \in \lie^\sigma_X$: We can write $f$ as
    \[
        f V = \frac{f + \overline{\sigma^* f}}{2} V + (-i) \left( \frac{i(f - \overline{\sigma^*f})}{2} V \right)
    \]
    which is a complex linear combination of complete real holomorphic fields. 
    Let $W$ be another complete real holomorphic field and $f \in \ker V, g \in \ker W$, $h \in (\ker V^2 \setminus \ker V) \cap \ker W$. Then $fV, gW, hfV, hgW \in \lie^\sigma_X$ and thus
    \begin{align*}
        [fV, hgW] - [hfV, gW] = fg V(h) W \in \lie^\sigma_X. 
    \end{align*}
    Thus the $\holo(X)$-submodule $E = \holo(X) V(h) W$ of $\vf(X)$ generated by a real holomorphic compatible pair $(V, W)$ is contained in $\lie^\sigma_X$. The case (2) in Definition \ref{flex-compatiblePair} (iii) is similar.

    Assume that the evaluation of the submodule $E$ at $p$ does not vanish. Indeed the set where the evaluation vanishes is a closed analytic  subvariety. By $\sigma$-holomorphic flexibility there is an automorphism in $\saut(X)$ taking $p$ out of this subvariety.
    
    By (3) and the Nakayama lemma there are $\Phi_1, \dots, \Phi_m \in (\saut(X))_p$ such that the stalk of $\bar{E}=\oplus_{j = 1}^m \Phi_j^* E$ at $p$ coincides with the stalk of the tangent sheaf $T_X$. 

    By its construction the set of global sections of $\bar{E}$ is a finitely generated $\holo(X)$-module of $\vf(X)$. Denote the generators by $V_1, V_2, \dots, V_m$. The equality of the stalks $\bar{E}_x = T_{X,x}$ is equivalent to whether the $V_i$'s span the tangent space $T_x X$, which is an open condition. 
    The analytic subset $Y$ of $X$ where $V_i$'s do not span can have countably many irreducible components $Y_1, Y_2, \dots$. There exists by Kutzschebauch \cite{MR3229363}*{Lemma 25} an equivariant automorphism $\Psi$ such that $\Psi(X \setminus Y) \cap Y_j \neq \emptyset$ for every $j = 1, 2, \dots$. Adding $\Psi^* \Phi_j^* E$'s to $\bar{E}$, we reduce the dimension of the singular variety and may proceed with the induction.
    %By the $\sigma$-holomorphic flexibility of $X$, there exist finitely many complete fields in $\vf^\sigma(X)$ spanning the tangent space $T_x X$ for all $x \in X$. The subgroup generated by the complex flows of these real holomorphic fields acts transitively on $X$. 
\end{proof}

\begin{proposition} \label{compPairs-G}
    Let $G$ be a linear algebraic group and $\sigma \colon G \to G$ the complex conjugation. Then 
    \begin{enumerate}
        \item there exists a compatible pair of real holomorphic vector fields if $G$ is different
    from $(\C^*)^m, m \in \N$ or $\C$.
        \item For $G \cong (\C^*)^m, m \ge 2$, there also exists  a semicompatible pair of real holomorphic vector fields. 
    \end{enumerate} 
\end{proposition}
\begin{proof}
    (1) The compatible pairs constructed in \cite{MR2385667}*{\S 3} for $G$ are real holomorphic with respect to $\sigma$. 

    (2) We consider the pair of complete real holomorphic vector fields $( z \frac{\partial}{\partial z}, w \frac{\partial}{\partial w})$ on $\C^*_{z} \times \C^*_w \times (\C^*)^{m-2}$. The linear span of the product of their kernels contains all rational polynomials.
\end{proof}

We will use the following result proved in \cite{Cotan}*{Proposition 6.4}.
\begin{lemma} \label{sigDP-GenSet}
    Let $G$ be a linear algebraic group and $\sigma \colon G \to G$ the complex conjugation.  
    Then
    \begin{enumerate}
        \item $G$ is $\sigma$-holomorphically flexible.
        \item For any nonzero $v \in T_e G \cap \mathrm{Fix}(d \sigma_e)$ the orbit of $v$ under the action of the isotopy group $(\saut(G))_e$ on $T_eG$ contains a basis of $T_e G$. 
    \end{enumerate}
    
\end{lemma}

\begin{proposition} 
    Let $G$ be a linear algebraic group whose connected components are different
    from $(\C^*)^m, m \in \N$ or $\C$ and let $\sigma \colon G \to G$ be the complex conjugation. Then $G$ has the $\sigma$-density property. 
\end{proposition}
\begin{proof}
    By Theorem \ref{sigma-DP}, Proposition \ref{compPairs-G} and Lemma \ref{sigDP-GenSet} the $\sigma$-density property of $G$ follows. 
\end{proof}

Thus by Theorem \ref{Global-Carleman} diffeomorphisms of such split real forms $G_\R$ are holomorphically approximable.

\begin{theorem} \label{G-Diffeo}
    Let $G$ be a linear algebraic group whose connected components are different
    from $(\C^*)^m, m \in \N$ or $\C$ and let $\sigma \colon G \to G$ be the complex conjugation. Denote by $G_\R = \mathrm{Fix}(\sigma)$ the split real form of $G$.
    
    Then for every diffeomorphism $\varphi \in \diff(G_\R)$ smoothly isotopic to $\id$, for any $k \in \N$ and any positive continuous function $\epsilon$ on $G_\R$, there exists a holomorphic automorphism $\Phi \in \saut(G)$ such that
	\[	
        d_k( j^k \Phi (p) - j^k \varphi(p) ) < \epsilon (p) \quad \text{ for all }\, p \in G_\R.	
    \]
\end{theorem}
\medskip 

\subsection{Approximation for Volume-preserving Diffeomorphisms}

\begin{proposition} \label{sigVDP-GenSet}
    Let $G$ be a linear algebraic group and $\sigma \colon G \to G$ the complex conjugation. Then for any nonzero $\alpha \in \wedge^2 (T_e G \cap \mathrm{Fix}(d \sigma_e))$ the orbit of $\alpha$ under the action of the isotopy group $(\saut(G))_e$ on $\wedge^2 T_e G$ contains a basis of $\wedge^2 T_e G$. 
\end{proposition}
\begin{proof}
    The proof is similar to that of Lemma \ref{sigDP-GenSet} (2) in \cite{Cotan}*{Proposition 6.4}. See also \cite{MR3492044}*{Lemma 5.1}.
\end{proof}

\begin{theorem} \label{sig-VDP}
    Let $X$ be a Stein manifold of dimension $n \ge 2$ and $\sigma \colon X \to X$ be an antiholomorphic involution such that the fixed-point set $X_\R=\mathrm{Fix}(\sigma) \neq \emptyset$. Let $\Omega$ be a holomorphic volume form on $X$ and assume $H^{n-1}(X, \C) = 0$ where $n = \dim X$. If
    \begin{enumerate}
        \item $X$ is $\sigma$-holomorphically volume flexible, and
        \item $X$ admits a semicompatible pair $(V,W)$ of complete real holomorphic volume-preserving vector fields, and
        \item there are $p \in X$ such that the orbit of $V(p) \wedge W(p) \in \wedge^2 T_p X$ under the action of the isotopy group $(\saut(X))_e$ contains a basis of $\wedge^2 T_p X$, 
    \end{enumerate}
     then $X$ has the $\sigma$-volume density property. 
\end{theorem}
\begin{proof}
    The proof is in the spirit of \cite{MR3492044} and the $\sigma$-equivariant part is similar to the proof of Theorem \ref{sigma-DP}. As is natural for volume-preservation, the module where we construct an $\holo(X)$-submodule out of the semicompatible pair $(V,W)$ is the module of holomorphic $(n-2)$-forms. 
    For $f \in \ker V, g \in \ker W$ by the formula 
    \begin{align*}
        i_{[fV, gW]} \Omega = d( fg \, i_V i_W \Omega )
    \end{align*}
    we get from the semicompatible pair an $\holo(X)$-submodule $E = I \, i_V i_W \Omega $ in $\Omega^{n-2} X$, where $I$ is the nontrivial ideal contained in $\overline{ \mathrm{Span}} (\ker V \cdot \ker W)$. 
    
    As in the proof of Theorem \ref{sigma-DP} by (1) we may assume that at point $p$ the evaluation of $E$ is nontrivial. By the tangential condition (3) we can enlarge $E$ such that its stalk at $p$ coincides with the stalk of $\Omega^{n-2} X$. Again the equality of the stalks is an open condition, since by the $\sigma$-holomorphic volume flexibility there are finitely many sections of $\wedge^{2} TX$ constructed from complete fields in $\vf^\sigma_\Omega(X)$ which span $\wedge^{2} TX$. Then we can enlarge the submodule so that after finitely many steps (of reducing the dimension of the singular set) we get the equality of stalks at all points $x \in X$. Hence we get all exact $(n-1)$-forms on $X$. 

    Since $H^{n-1}(X, \C)=0$, we have generated all closed $(n-1)$-forms on $X$ and thus by the nondegeneracy of $\Omega$ all volume-preserving vector fields on $X$. 
\end{proof}

\begin{remark} \label{sig-VDP-H(n-1)≠0}
    For the $\sigma$-volume density property, the vanishing condition $H^{n-1}(X, \C)=0$ can be replaced by the following:
    \begin{center}
        \it There exist finitely many complete vector fields $V_1, V_2, \dots$ in $\vf_\Omega^\sigma(X)$ such that $[i_{V_1}\Omega], [i_{V_2}\Omega], \dots$ span the holomorphic de Rham cohomology group $H^{n-1}(X)$. 
    \end{center}
\end{remark}

\begin{proposition} \label{G-sig-VDP}
    Let $G$ be a linear algebraic group equipped with the left-invariant volume form $\Omega$ and $\sigma \colon G \to G$ the complex conjugation. Then $G$ has the $\sigma$-volume density property. 
\end{proposition}
\begin{proof}
    By \cite{Cotan}*{Lemma 6.3} there exists a basis of $T_e G$ which induces complete real holomorphic vector fields on $G$. 
    Every such left-invariant vector field on $G$ preserves the left-invariant form $\Omega$. Thus $G$ has the $\sigma$-holomorphic volume flexibility. 
    \medskip 

    Since all connected components of $G$ are isomorphic as varieties we can assume that $G$ is connected.
    Consider the Mostow decomposition $G = L \ltimes R_u$ where $L$ is an affine maximal closed reductive algebraic subgroup of $G$ and $R_u$ the unipotent radical of $G$. As affine variety $G \cong L \times R_u \cong L \times \C^k$ for some $k \in \N$. If $k \ge 2$, then $H^{n-1}(G, \C)=0$ by the Künneth formula. For $k = 1$, the volume form $\Omega = \Omega_L \times dz$ and we can take $V_1 = \partial/\partial z$ so that $[i_{V_1} \Omega] = [\Omega_L]$ spans $H^{n-1}(G)$. 

    When $k=0$, the reductive $G=L$ has center $Z \cong (\C^*)^m$ and the semisimple part $S$, which has $H^{n-1}(S, \C) = 0$. Up to a finite central normal subgroup $G$ is a direct product $Z \times S$, in the case $m \ge 1$ we can use the corresponding $(n-1)$-forms of $V_i = z_i \partial/\partial z_i$, $i = 1, \dots, m$ to generate $H^{n-1}(G)$ similarly as for $k =1$ above. 
    
    It remains to consider when $G$ is semisimple, which has the maximal compact subgroup $K$ as a deformation retract. Since $K$ has finite fundamental group, by Poincar{\'e} duality $H^{n-1}(G) = H^{n-1}(K) = H_1(K) = 0$. 
    
    \medskip 
    For $\dim (G) \ge 2$, by Proposition \ref{compPairs-G}, \ref{sigVDP-GenSet}, Theorem \ref{sig-VDP} and Remark \ref{sig-VDP-H(n-1)≠0} the $\sigma$-volume density property of $G$ follows. When $\dim(G) =1$, $G= \C$ or $\C^*$ has the $\sigma$-volume density property. 
\end{proof}

\begin{lemma} \label{LAG-Morse-rho}
    Let $G$ be a linear algebraic group equipped with the left-invariant volume form $\Omega$ and $\sigma \colon G \to G$ the complex conjugation. Then there exist a real holomorphic embedding $G \hookrightarrow \C^q$ and $r \ge 0$  such that the restriction of $\rho$ to $G_\R$ has only critical points in the interior of $Z^\R_r = \{\, x \in G_\R: \rho(x) \le r \,\}$ (from Definition \ref{sublevelset}). 
\end{lemma}
\begin{proof}
    Every linear algebraic group over $\C$ is a subgroup of $GL_n(\C)$ for some $n \in \N$, which can be embedded in $\C^{n^2+1}$. By our choice of the real structure $\sigma$ this embedding $\jmath \colon G \hookrightarrow \C^{n^2+1}$ is real holomorphic and the real form $G_\R = \mathrm{Fix}(\sigma)$ is embedded into $\R^{n^2+1}$.

    The reference point for the distance function $\rho$ on $\C^{n^2+1}$ can be taken to be in the complement $\R^{n^2+1} \setminus \jmath(G_\R)$, see the induction base in the proof of Theorem \ref{Global-Carleman-vol}. By taking a generic point $p_0$, we get an exhaustion function $\rho$ which is a Morse function with isolated critical points (see Gullemin--Pollack \cite{MR2680546}*{p.\@ 43}) . Since the manifold $G_\R$ is algebraic, the restriction of $\rho$ to $G_\R$ has only finitely many critical points which are thus contained in a bounded sublevel set $Z^\R_r$ for some $r > 0$.   
\end{proof}

\begin{theorem} \label{G-vol-Diffeo}
    Let $G$ be a linear algebraic group equipped with the left-invariant volume form $\Omega$ and $\sigma \colon G \to G$ the complex conjugation. Denote by $G_\R = \mathrm{Fix}(\sigma)$ the split real form of $G$. 

    Then for every volume-preserving diffeomorphism $\varphi \in \diff_{\Omega_\R,0}(G_\R)$ smoothly isotopic to $\id$, for any $k \in \N$ and any positive continuous function $\epsilon$ on $G_\R$, there exists a holomorphic automorphism $\Phi \in \saut_\Omega(G)$ such that
	\[	
        d_k( j^k \Phi (p) - j^k \varphi(p) ) < \epsilon (p) \quad \text{ for all }\, p \in G_\R.	
    \]
\end{theorem}
\begin{proof}
    Up to finite covering the connected components of the manifold $G_\R$ are diffeomorphic to $\R^j \times \R_{>0}^k \times S_\R$. The semisimple part $S_\R$ as a manifold is diffeomorphic to $\R^r \times K$, where $K$ is the maximal compact subgroup and $r$ the rank of the semisimple group $S_\R$. The codimension of $K$ in $S_\R$ is given by $(\dim S_\R + r)/2 >1$, thus $H^{\dim S_\R - 1}(S_\R) = H^{\dim S_\R-1}(K) = 0$. 

    By Lemma \ref{LAG-Morse-rho}, Proposition \ref{G-sig-VDP} and Theorem \ref{Global-Carleman-vol} the assertion follows.
\end{proof}
\medskip 

\subsection{Approximation for Hamiltonian Diffeomorphisms}

In this subsection we collect known holomorphic approximation for Hamiltonian diffeomorphisms of noncompact totally real submanifolds from our earlier works which are based on the method outlined in Section \ref{sec: Ham-approximation}.

Let $G_\R$ be a real Lie group with Lie algebra $\mathfrak{g}_\R$ which we identify with the tangent space at the identity $T_eG_\R$. The conjugation of $G_\R$ on itself induces the adjoint action on $\mathfrak{g}_\R$ and by duality the coadjoint action $\mathrm{Ad}^*$ of $G_\R$ on the dual $\mathfrak{g}^*_\R$ defined as 
\[
    (\mathrm{Ad}^*(g) \xi, v) = (\xi, \mathrm{Ad}(g)v) \text{ for } g \in G_\R, \xi \in \mathfrak{g}^*_\R, v \in \mathfrak{g}_\R, 
\]
where $(\cdot, \cdot)$ is the pairing between $\mathfrak{g}^*_\R$ and $\mathfrak{g}_\R$. 

Coadjoint orbits are the $G_\R$-orbits of elements of $\mathfrak{g}^*_\R$ and each of them carries a canonical symplectic structure. Let $O_\R$ be a coadjoint orbit and $v \in \mathfrak{g}_\R$, the smooth vector field 
\[
    X_v (\xi) = \restr{\frac{d}{dt}}{t= 0} \mathrm{Ad}^*(\exp(t v)) \xi
\]
is the infinitesimal action induced by the one parameter subgroup $\{ \, \exp(tv), t \in \R \, \}$ of $G_\R$. By transitivity of $G_\R$ on $O_\R$ the vector fields $\{\, X_v(\xi), v \in \mathfrak{g}_\R \,\}$ span the tangent space $T_\xi O_\R$ and the symplectic form on the coadjoint orbit $O_\R$ is given by
\[
    \omega (X_u(\xi), X_v(\xi) ) = \xi ([u,v]) \text{ for } u, v \in \mathfrak{g}_\R. 
\]
In the same way we have complex coadjoint orbits of the complex Lie group $G$ with $G_\R$ as a real form and there is a holomorphic symplectic form on the complex $G$-orbit $O$ in the complexification $\mathfrak{g}^*=\mathfrak{g}^*_\R \otimes \C$. 

The first result is  from \cite{MR4423269} where holomorphic approximations for diffeomorphisms of real forms of closed coadjoint orbits of complex Lie groups were shown. 

\begin{theorem}
    Let $O_\R$ be a real coadjoint orbit of $G_\R$. If the complex coadjoint orbit $O$ is closed in  $\mathfrak{g}^*$ and the first de Rham cohomology of $O_\R$ is trivial, then the Hamiltonian diffeomorphism group of $O_\R$ is holomorphically approximable. 
\end{theorem}

The second approximation is for the Calogero--Moser spaces $\camo$, a smooth affine symplectic variety of dimension $2n$ given by the categorical quotient 
\begin{align*}
    \{\, (X, Y) \in \mathrm{Mat}(n \times n, \C) \times \mathrm{Mat}(n \times n, \C): \mathrm{rank}([X, Y] + \id) = 1 \,\} /\!/ \mathrm{PGL}_n(\C), 
\end{align*}
where the $\mathrm{PGL}_n(\C)$-action is $g \cdot (X, Y) = (gXg^{-1}, gYg^{-1})$. %The symplectic form $\omega$ on $\camo$ is induced by $d X \wedge d Y$. 
The antiregular involution $\sigma \colon \camo \to \camo, (X, Y) \mapsto (\bar{X}^{t}, \bar{Y}^{t})$ fixes a totally real manifold $\camo^\R = \mathrm{Fix}(\sigma)$. By \cite{MR4868759} $\camo$ has the $\sigma$-Hamiltonian density property.

\begin{theorem}[\cite{CaloCarleman}]
    The Hamiltonian diffeomorphism group $\ham_{\omega_\R}(\camo^\R)$ of $\camo^\R$ is holomorphically approximable.  
\end{theorem}

The latest result is the following proved in \cite{Cotan}.

\begin{proposition} \label{holo-tauFlex-to-SDP}
    Let $M$ be a Stein manifold and $\tau \colon M \to M$ an antiholomorphic involution such that the fixed-point set $M_\R=\mathrm{Fix}(\tau) \neq \emptyset$. If 
    \begin{enumerate}[label=(\roman*)]
        \item $M$ is Stein and $\tau$-holomorphically flexible, and
        \item there is a point $x \in M_\R$ such that the induced action by the isotopy group $(\taut(M))_{x}$ on $T_{x}^\R M_\R \setminus \{0\}$ is transitive,
    \end{enumerate}
    then its holomorphic cotangent bundle $T^* M$ has the $\tilde{\tau}$-Hamiltonian density property and the $\tilde{\tau}$-symplectic density property, where $\tilde{\tau}$ is the antiholomorphic involution on $T^*M$ induced by $\tau$. 
\end{proposition}

As application of our Theorem \ref{Global-Carleman-ham} one proves the following.

\begin{theorem}[\cite{Cotan}] \label{thm: homoSp-Ham}
Let $G_\R$ be a real Lie group, $H_\R$ a closed subgroup and $G$, $H$  the universal complexification of $G_\R$, $H_\R$ respectively. Assume that $i \colon G_\R \to G$ is injective, $G$ linear algebraic and $G/H$ affine of complex dimension $\ge 2$. Denote by $\tau \colon G/H \to G/H$ the antiholomorphic involution induced by the involution $\tau_G$ of $G$ on $G/H$. 

Then for every Hamiltonian diffeomorphism $\varphi \in \diff_{\omega_\R}(T^*_\R (G_\R/H_\R))$, for any $k \in \N$ and any positive continuous function $\epsilon$ on $T_\R ^*(G_\R/H_\R)$, there exists a holomorphic symplectomorphism $\Phi \in \aut^{\tilde{\tau}}_{\omega}(T^* (G/H))$ such that
	\[	d_k( j^k \Phi (p) - j^k \varphi(p) ) < \epsilon (p) \quad \text{ for all }\, p \in T_\R^* (G_\R/H_\R).	\]
\end{theorem}

\section*{Funding}
The first author was partially supported by National Key R\&D Program of China [2021YFA1003100], National Natural Science Foundation of China [12471079, 12525104] and the Fundamental Research Funds for the Central Universities.  
The second author was funded by Schweizerischer Nationalfonds Postdoc.Mobility [P500-$2\_239113$]. The third author was partially supported by Schweizerischer Nationalfonds [10005432].

\end{document}